\pdfoutput=1

\documentclass[leqno,12pt]{amsart}

\usepackage{amsmath}

\usepackage{amssymb}
\usepackage{graphicx}

\newtheorem{thm}{Theorem}

\newtheorem{prop}[thm]{Proposition}
\newtheorem{lemma}[thm]{Lemma}

\theoremstyle{definition}
\newtheorem{defn}{Definition}

\theoremstyle{remark}
\newtheorem{remark}{Remark}

\numberwithin{equation}{section}

\def\C{\mathbb{C}}
\def\R{\mathbb{R}}

\begin{document}
\title[]{A conformally invariant variational problem for time-like curves}

\author{Olimjon Eshkobilov}
\address{(O. Eshkobilov) Dipartimento di Matematica, Università di Torino,
Via Carlo Alberto, 10 - I-10123 Torino, Italy}
\email{olimjon.eshkobilov@edu.unito.it}

\author{Emilio Musso}
\address{(E. Musso) Dipartimento di Scienze Matematiche, Politecnico di Torino,
Corso Duca degli Abruz\-zi 24, I-10129 Torino, Italy}
\email{emilio.musso@polito.it}


\thanks{Authors partially supported by
PRIN 2010-2011 ``Variet\`a reali e complesse: geometria, to\-po\-lo\-gia e analisi ar\-mo\-ni\-ca'' and by
GNSAGA of INDAM.}

\subjclass[2000]{53C50, 53A30}

\date{Version of \today}


\keywords{Conformal Lorenztian geometry, timelike curves, conformal invariants, Einstein universe, conformal Lorentzian compactification, conformal strain}

\begin{abstract}
We study the conformally invariant variational problem for time-like curves in the $n$-dimensional Einstein universe defined by the conformal strain functional. We prove that the stationary curves are trapped into an Einsetin universe of dimension $2$, $3$ or $4$. We study the linearly-full stationary curves in a four-dimensional Einstein universe and we show that they can be integrated by quadratures in terms of elliptic functions, elliptic integrals and Jacobi's theta functions.
\end{abstract}

\maketitle

\section{Introduction}\label{s:intro}
Lorentzian conformal geometry began in $1918$ with the seminal work of H.Weyl \cite{W}. In the 1980s, it has been considerably developed in connection with the twistor approach to gravity by Penrose and Rindler \cite{PR} and it is widely used in the cyclic cosmological models in general relativity \cite{Pn1,Pn3, Tod}. It is also related to regularization of the Kepler problem \cite{GS} and Lie sphere geometry \cite{Blaschke, Ce, JMN}. The purpose of the present paper is to investigate a variational problem for time-like curves in the conformal compactification $\mathcal{E}^{1,n-1}$ of the Minkowski n-space. In the literature, $\mathcal{E}^{1,n-1}$ is often referred to as the {\it compact $n$-dimensional Einstein universe} \cite{BCDG,Fr,GS}. The choice of working within the compact model is motivated by the fact that, in this case, the conformal group is a group of matrices. This greatly simplifies the computational aspects. However, our considerations apply as well in the non-compact case. It 'should be noted that any Lorentzian space-form or any Friedmann--Lemaitre--Robertson--Walker cosmological model can be realized via a conformal embedding, as an open domain of $\mathcal{E}^{1,n-1}$ \cite{HE}.

\noindent Lorentzian conformal geometry shares, at least on a formal level, many common features, with Riemannian conformal geometry, a classical topic in differential geometry. The Lorentzian conformal geometry of the Einstein universe can be developed in analogy with the conformal geometry of a sphere (M\"obius geometry). For instance, the existence of a canonical conformal invariant arc-element along a generic curve in $S^n$ is a well known fact that goes back to the works \cite{F, Ha, L, T, V} of Fialkov, Haantjes, Liebmann, Takasu and Vessiot, all published in the first decades of the past century. In \cite{DMN} the construction of the conformal arc-element has been extended to time-like curves in a $3$-dimensional Einstein universe. In this paper we show how to define a conformal invariant arc-element for a time-like curve in an Einstein universe of arbitrary dimension. The integral of the conformal arc-element defines the simplest conformal invariant variational problem on the space of generic time-like curves of $\mathcal{E}^{1,n-1}$, called the conformal strain functional. This is the Lorentzian counterpart of the conformal arc-length functional in M\"obius geometry \cite{LO2010, LS, MMR, M1, M2, MN} and generalizes the homonymous functional for time-like curves in $\mathcal{E}^{1,2}$, previously considered in \cite{DMN}. Proceeding in analogy with \cite{MMR} and using the method of moving frames we deduce the variational equations satisfied by the stationary curves, referred to as conformal world-lines. We prove that conformal world-lines in $\mathcal{E}^{1,n-1}$ are trapped into a totally geodesic Einstein universe of dimension $2$, $3$ or $4$. This is the Lorentzian analogue of a similar result in M\"obius geometry \cite{MMR}. The $2$-dimensional case is rather trivial since the trajectories are orbits of $1$-parameter groups of conformal transformations. The $3$-dimensional case is more challenging and it has been partially investigated in \cite{DMN}. Here, we focus on the world-lines of $\mathcal{E}^{1,3}$ which are not trapped in a lower dimensional totally umbilical Einstein universe (linearly-full world lines). We prove that their trajectories can be explicitly integrated by quadratures in terms of elliptic functions and elliptic integrals.

The paper is organized as follows. In Section \ref{s1}, we collect from the literature few basic facts about conformal Lorentzian geometry \cite{BCDG,GS} and we reformulate in the Lorentzian context the classical approach to the conformal geometry of curves in the M\"obius space \cite{F,Ha,Mon,SS,S,T,Thomsen}. We define the conformal strain and the conformal line-element, which is the Lorentzian analogue of the conformal arc element of a curve in $S^n$ \cite{LO2010, MMR, Mon, M1, M2, MN}. In Section \ref{s2}, we use the moving frame method  to compute the Euler-Lagrange equations of the conformal strain functional (Theorem \ref{thmA}). Consequently we show that the
trajectory of a conformal world-line lies in a totally umbilical Einstein universe of dimension $2$, $3$ or $4$ (Theorem \ref{ThmB}). In Section \ref{s3} we study linearly full conformal world-lines of a four-dimensional Einstein universe. Given such a world-line we build a canonical lifting to the conformal group, the canonical conformal frame and we define the three conformal curvatures (Proposition \ref{propuniq}). Then we use the variational equation to show that the curvatures are either constants or else can be expressed in terms of Jabobi's elliptic functions  (Proposition \ref{ConfCurv}). In the first case the trajectory is an orbit of a $1$-parameter group of conformal transformations. Leaving aside the constant curvature case, we prove that the conformal equivalence classes of world lines can be parameterized by three real parameters (Proposition \ref{PHRS}). Then we define the momentum operator, that is an element of the Lie algebra of the conformal group, intrinsically defined by the world-line. The existence of the momentum is a consequence of the conformal invariance of the functional and of the N\"other conservation theorem (Remark \ref{ThEx}). We prove that the momentum operator is either a regular or an exceptional element of the conformal Lie algebra (Proposition \ref{CT}). Then we define the integrating factors of an eigenvalue of the momentum and we build the principal vectors of the eigenvalues (Proposition \ref{PV}). The explicit computation of the integrating factors is rather technical and is considered in the Appendix. In Section $5$ we use the integrating factors and the principal vectors
to integrate by quadratures the trajectory of a linearly full conformal world-line with non-constant curvatures (Theorem \ref{ThmC} and Theorem \ref{ThmD}). At the end of the section we briefly comment the theoretical aspects underlying the integration by quadratures, we explain why a linearly full world-line with non-constant curvatures, can't be closed and we show that the trajectory of a world-line is invariant by the action of
an infinite cyclic group of conformal transformations.

\section{Conformal geometry of a time-like curve}\label{s1}

\subsection{The Einstein universe and its restricted conformal group} We denote by
$\mathcal{E}^{1,n-1}$ the n-dimensional sub-manifold of $\R^{n+2}, n>1$, defined by the equations $x_0^2+x_1^2=1$ and $x_2^2+\dots + x_{n+1}^2=1$. As a manifold, $\mathcal{E}^{1,n-1}$ is the Cartesian product $S^1 \times S^{n-1}$. The restriction of the quadratic form
$$g=-dx_0^2-dx_1^2+dx_2^2+\dots + dx_{n+1}^2$$
induces a Lorentzian pseudo-metric $g_{\textsl{e}}$ on $\mathcal{E}^{1,n-1}$. The normal bundle of $\mathcal{E}^{1,n-1}$ is spanned by the restrictions of the vector fields $\mathbf{n}_1=x_0\partial_{x_0}+x_1\partial_{x_1}$ and $\mathbf{n}_2=x_2\partial_{x_2}+\dots +x_{n+1}\partial_{x_{n+1}}$. Thus, contracting $dx_0\wedge\dots dx_{n+1}$ with $\mathbf{n}_1$ and $\mathbf{n}_2$, we get a volume form on $\mathcal{E}^{1,n-1}$ which in turn defines an orientation. The vector field $-x_1\partial_{x_0}+x_0\partial_{x_1}$ is tangent to $\mathcal{E}^{1,n-1}$ and induces a unit time-like vector field  on $\mathcal{E}^{1,n-1}$.
We time-orient $\mathcal{E}^{1,n-1}$ by requiring that such a vector field is future-oriented.

\begin{defn}
The Lorentzian manifold $(\mathcal{E}^{1,n-1},g_{\textsl{e}})$, with the above specified orientation
and time-orientation, is called the {\it $n$-dimensional Einstein universe}.
\end{defn}

\noindent In order to describe the conformal geometry of $\mathcal{E}^{1,n-1}$ it is convenient to consider in $\R^{n+2}$ the coordinates
$$y_0=\frac{1}{\sqrt{2}}(x_0+x_{n+1}),\quad y_1=x_1, \dots , y_n=x_n,\quad y_{n+1}=\frac{1}{\sqrt{2}}(-x_0+x_{n+1}).$$
The corresponding basis of $\R^{n+2}$ is denoted by $(E_0,\dots, E_{n+1})$ and is said the {\it standard (light-cone) basis} of $\R^{n+2}$. With respect to $(y_0,\dots , y_{n+1})$ the scalar product associated to $g$ can be written as
\begin{equation}\label{sp}\langle Y,Y'\rangle=-(y_0y'_{n+1}+y_{n+1}y'_0)-y_1y'_1+\sum_{j=2}^{n}y_jy'_j.\end{equation}
In addition, $dV=dy_0\wedge \dots \wedge y_{n+1}$ is a positive volume form. From now on we will use the light-cone coordinates and the elements $V\in \R^{n+2}$ are thought of as column vectors constructed from the components of $V$ with respect to $(E_0,\dots ,E_{n+1})$. For each $V\in \R^{n+2}$, $V\neq 0$, we denote by $[V]$ the oriented line spanned by $V$. The map
$ V\in \mathcal{E}^{1,n-1}\to [V]$ allow us to identify  $\mathcal{E}^{1,n-1}$ and the manifold of the isotropic oriented lines through the origin of $\R^{n+2}$ (null rays). Using such an identification, the connected component of the identity $\mathrm{A}^{\uparrow}_+(2,n)$ of the pseudo-orthogonal group of (\ref{sp}) acts transitively and effectively on the left of $\mathcal{E}^{1,n-1}$ by $\mathbf{X}[V]=[\mathbf{X}\cdot V]$. The action preserves the oriented, time-oriented conformal Lorentzian structure of $\mathcal{E}^{1,n-1}$.
It is a classical result that, if $n>2$, every restricted conformal transformation of $\mathcal{E}^{1,n-1}$ is induced by a unique element of $\mathrm{A}^{\uparrow}_+(2,n)$ \cite{DNF,Fr}. For this reason, we call $\mathrm{A}^{\uparrow}_+(2,n)$ the {\it restricted conformal group of the $n$-dimensional Einstein universe}.

\begin{remark}{To distinguish the connected component of the identity we proceed as in \cite{GS} : we consider the cone $\mathcal{C}\subset \bigwedge^2(\R^{n+2})$ of all isotropic bi-vectors, ie the non-zero decomposable elements $V\wedge W$ of $\bigwedge^2(\R^{n+2})$ such that $\langle V,V\rangle = \langle W,W\rangle = \langle V,W\rangle = 0$. The function
$$\mathfrak{V}: V\wedge W \in \mathcal{C} \to dV(V,W,E_1,\dots E_n,E_{n+1}-E_0)\in \R$$
never vanishes and the half cones
$$\mathcal{C}_+=\{V\wedge W\in \mathcal{C} : \mathfrak{V}(V\wedge W)>0\},\quad  \mathcal{C}_-=\{V\wedge W\in \mathcal{} : \mathfrak{V}(V\wedge W)<0\}$$
are the two connected components of $\mathcal{C}$. Then, $\mathrm{A}^{\uparrow}_+(2,n)$ is the group of all pseudo-orthogonal matrices $\mathbf{B}$ of the scalar product (\ref{sp}) such that $\mathrm{det}(\mathbf{B})=1$ and $\mathbf{B}\cdot \mathcal{C}_+=\mathcal{C}_+$.}\end{remark}

\noindent We put
\begin{equation}\label{mm}\mathtt{m}=(\mathtt{m}_{ji}),\quad \mathtt{m}_{ji}=\langle \mathrm{E}_j,\mathrm{E}_i\rangle,\quad i,j=0,\dots, n+1.\end{equation}
Then the column vectors $B_0,\dots B_{n+1}$ of a matrix $\mathbf{B}\in \mathrm{A}^{\uparrow}_+(2,n)$ constitute a {\it light-cone basis} of $\R^{n+2}$, ie a positive-oriented basis such that
$$\langle B_i,B_j\rangle = \mathtt{m}_{ji}, \quad i,j=0,\dots n+1,\quad B_0\wedge(B_1+B_2)\in \mathcal{C}_+.$$
This allow us to identify $\mathrm{A}^{\uparrow}_+(2,n)$ and the manifold of all light-cone basis of
$\R^{n+2}$. We use the notation $\R^{2,n}$ to denote $\R^{n+2}$ equipped with the scalar product $(\ref{sp})$, the volume form $dV$ and the positive half cone $\mathcal{C}_+$. Differentiating the $\R^{2,n}$-valued maps
$$\mathcal{B}_j:B \in \mathrm{A}^{\uparrow}_+(2,n)\to B_j\in \R^{2,n},\quad j=0,\dots,n+1,$$
yields
$$d\mathcal{B}_j=\sum_{i=0}^{n+1}\mu_j^i \mathcal{B}_i,$$
where $\mu^i_j$ are left-invariant 1-forms. The conditions $\langle B_j,B_i\rangle= \mathtt{m}_{ji}$ imply that $\mu =(\mu^i_j)$ takes values in the Lie algebra
$$\mathfrak{a}(2,n)=\{X\in \mathfrak{gl}(n+2,\R) / ^tX\cdot \mathtt m + \mathtt m\cdot X=0\}.$$
of $\mathrm{A}^{\uparrow}_+(2,n)$. Consequently we can write
$$\mu = \left(
          \begin{array}{cccc}
            \mu^0_0 & -\mu^1_{n+1} & \mu_{n+1} & 0 \\
            \mu^1_0 & 0 & ^t\mu_1 & \mu^1_{n+1} \\
            \mu_0 & \mu_1 & \widetilde{\mu} & \mu_{n+1} \\
            0 & -\mu^1_0 & ^t\mu_0 & -\mu^0_0 \\
          \end{array}
        \right),$$
where
$$\mu_0=^t(\mu^2_0,\dots, \mu^n_0),\quad \mu_1=^t(\mu^2_1,\dots, \mu^n_1),\quad
\mu_{n+1}=^t(\mu^2_{n+1},\dots, \mu^n_{n+1})$$
and $^t\widetilde{\mu}+\widetilde{\mu}=0$. The left-invariant 1-forms $\mu^0_0$, $\mu^j_0$, $\mu^j_{n+1}$, $j=1,\dots, n$ and $\mu^i_j$, $1\le i<j=1,\dots n$ are linearly independent and span the dual of the Lie algebra $\mathfrak{a}(2,n)$. They satisfy the {\it Maurer-Cartan equations}
\begin{equation}\label{MC}d\mu^i_j=-\sum_{k=0}^{n+1}\mu^i_k\wedge \mu^k_j,\quad  i,j=0,\dots, n+1.\end{equation}

\begin{remark}
Let $\mathbb{M}^{1,n-1}$ be {\em Minkowski n-space}, i.e., the affine space $\R^n$ with the Lorentzian scalar product
$$(\mathbf{p},\mathbf{q})=-p_1q_1+p_2q_2+\dots + p_nq_n.$$
The map
\begin{equation}\label{j}
\mathbf{j}(\mathbf{p})= [^t(1,p_1,\dots,p_n,\frac{1}{2}(\mathbf{p},\mathbf{p}))]\in \mathcal{E}^{1,n-1}
\end{equation}
is a conformal embedding whose image is said the {\it Minkowski-chamber} of $\mathcal{E}^{1,n-1}$
Let $\mathrm{P}^{\uparrow}_+(1,n-1)=\mathbb{M}^{1,n-1}\rtimes \mathrm{SO}^{\uparrow}_+(1,n-1)$
be the restricted Poincar\'e group of $\mathbb{M}^{1,n-1}$. For each $(\mathbf{p},\mathrm{L})\in \mathrm{P}^{\uparrow}_+(1,n-1)$ we put $^{*}\mathbf{p}=(-p_1,p_2,\dots, p_n)$. The matrix
$$\mathbf{B}(\mathbf{p},\mathrm{L})=\left(
                      \begin{array}{ccc}
                        1 & 0 & 0 \\
                        \mathbf{p} & \mathrm{L} & 0 \\
                        \frac{1}{2}^{*}\mathbf{p}\cdot \mathbf{p} & ^{*}\mathbf{p}\cdot \mathrm{L} & 1 \\
                      \end{array}
                    \right)
$$
belongs to $\mathrm{A}^{\uparrow}_+(2,n)$ and
\begin{equation}\label{J}\mathbf{J}:(\mathbf{p},\mathrm{L})\in \mathrm{P}^{\uparrow}_+(1,n-1)\to \mathbf{B}(\mathbf{p},\mathrm{L}))\in \mathrm{A}^{\uparrow}_+(2,n)\end{equation}
is a faithful representation. In a similar way, one can build conformal embeddings of the de-Sitter or anti de-Sitter $n$-spaces into $\mathcal{E}^{1,n-1}$. Also the Robertson-Walker $n$-spaces can be conformally mapped as open sub-manifolds of the Einstein universe \cite{HE}.\end{remark}

\begin{defn}{A $(2+h)$-dimensional vector subspace $\mathbb{W}\subset \R^{2,n}$, $h=1,...,n-1$ is said {\it Lorentzian} if the restriction of $\langle -, -\rangle $ to $\mathbb{W}$ is non-degenerate, of signature $(2,h)$. The set $\mathcal{E}(\mathbb{W})$ of all null rays belonging to $\mathbb{W}$ is a $h$-dimensional totally umbilical Lorentzian sub-manifold of $\mathcal{E}^{1,n-1}$. We call $\mathcal{E}(\mathbb{W})$ a {\it $h$-dimensional Lorentzian cycle} of $\mathcal{E}^{1,n-1}$. If we equip $\mathcal{E}(\mathbb{W})$ with the induced conformal structure we get a Lorentzian conformal space equivalent to a $h$-dimensional Einstein universe.}\end{defn}

\subsection{Time-like curves} Let $\gamma : {I}\subset \R\to \mathcal{E}^{1,n-1}$ be a parameterized time-like curve, $n\ge 3$. A {\it null-lift} is a map $\Gamma :I\to \R^{2,n}$ such that $\gamma(t)=[\Gamma(t)]$, for every $t\in I$. Since $\gamma$ is a time-like immersion, its null-lifts satisfy
$$\Gamma\wedge \Gamma'|_t\neq 0,\quad \langle \Gamma|_t,\Gamma|_t=\langle \Gamma|_t,\Gamma'|_t\rangle =0,\quad
\langle \Gamma'|_t,\Gamma'|_t\rangle < 0,$$
for every $t\in I$. For each $k=1,...,n+1$, the {\it $k$-th osculating space} of $\gamma$ at $\gamma(t)$ is the linear subspace $\mathcal{T}^k(\gamma)|_t\subset \R^{2,n}$ spanned by the vectors $\Gamma|_t,\Gamma'|_t,\dots ,\Gamma^k|_t$. Obviously the definition does not depend on the choice of the null lift. Putting together all the osculating spaces we get the k-th osculating sheaf
$$\mathcal{T}^{k}(\gamma)=\{(t,V)\in I\times \R^{2,n} : V\in \mathcal{T}^k(\gamma)|_t\}.$$

\begin{lemma}{$\mathcal{T}^2(\gamma)|_t$ is a $3$-dimensional Lorentzian subspace of $\R^{2,n}$, for every $t\in I$.}\end{lemma}
\begin{proof}{Choose $\Gamma$ so that $\langle \Gamma',\Gamma'\rangle = -1$. Differentiating $\langle \Gamma,\Gamma'\rangle = 0$ and $\langle \Gamma',\Gamma'\rangle = -1$
we get $\langle \Gamma,\Gamma''\rangle = 1$ and $\langle \Gamma',\Gamma''\rangle = 0$. This implies that $\Gamma|_t,\Gamma'|_t$ and $\Gamma''|_t$ are linearly independent, for each $t\in I$. We put
$$A_1= \frac{1}{2}(1+\langle \Gamma'',\Gamma''\rangle)\Gamma - \Gamma'',\quad A_2=\Gamma' \quad A_3=\frac{1}{2}(1-\langle \Gamma'',\Gamma''\rangle)\Gamma + \Gamma''.$$
Then, $(A_1|_t,A_2|_t,A_3|_t)$ is a basis of $\mathcal{T}^2(\gamma)|_t$ such that
\[
\begin{split}
 \langle A_i|_t,A_j|_t\rangle &= 0,\quad i\neq j,\\
\langle A_1|_t,A_1|_t\rangle &= \langle A_2|_t,A_2|_t\rangle=-\langle A_3|_t,A_3|_t\rangle = -1.
   \end{split}
   \]
This yields the result.}\end{proof}

\noindent The Lemma implies that $\mathcal{T}^k(\gamma)|_t$ is a Lorentzian subspace of $\R^{2,n}$, for every $k\ge 2$ and every $t\in I$. The Lorentzian cycle $\mathcal{E}(\mathcal{T}^k(\gamma)|_t)$ is denoted by $\mathcal{E}^k(\gamma)|_t$. We call $\mathcal{E}^k(\gamma)|_t$ the {\it $k$-th osculating cycle} of $\gamma$ at $\gamma(t)$. The orthogonal complement $\mathcal{N}^k(\gamma)|_t$ of $\mathcal{T}^k(\gamma)|_t$ is a space-like vector subspace such that $\R^{2,n}=\mathcal{T}^k(\gamma)|_t \oplus \mathcal{N}^k(\gamma)|_t$.
We call $\mathcal{N}^k(\gamma)|_t$ the {\it $k$-th normal space} at $\gamma(t)$ and we define the {\it k-th normal sheaf} by
$$\mathcal{N}^{k}(\gamma)=\{(t,V)\in I\times \R^{2,n} : V\in \mathcal{N}^k(\gamma)|_t\}.$$
The orthogonal projection of $\R^{n+2}$ onto $\mathcal{N}^k(\gamma)|_t$ is denoted by $pr_{(k)}|_t$.

\begin{defn}{The {\it conformal strain} of a time-like curve is the quartic differential
\begin{equation}\label{strain-density}\mathcal{Q}_{\gamma}=
\frac{\langle pr_{(2)}(\Gamma'''),
pr_{(2)}(\Gamma''')\rangle}{|\langle \Gamma',\Gamma'\rangle|}dt^4
\end{equation}}\end{defn}

\noindent The conformal strain is independent on the choice of the null lift and, in addition, if $\gamma$ and $\widetilde{\gamma}$ are two {\it equivalent time-like curves}\footnote{ie, $\widetilde{\gamma}=B\cdot (\gamma\circ h)$, where $B\in \mathrm{A}^{\uparrow}_+(1,n+1)$ and $h$ is a change of parameter.} then $\mathcal{Q}_{\widetilde{\gamma}}=h^*(\mathcal{Q}_{\gamma})$.

\begin{defn}{If $\mathcal{Q}(\gamma)|_t=0$, the point $\gamma(t)$ is said to be a {\it conformal vertex}. A time-like curve without conformal vertices is said {\it generic}. If $\mathcal{Q}_{\gamma}=0$, the curve is said {\it totally degenerate}.}\end{defn}

\begin{remark}{If $\gamma$ is totally degenerate then $\mathcal{E}^1(\gamma)|_t$ is constant and the trajectory of $\gamma$ is contained in a 1-dimensional conformal cycle.}\end{remark}

\noindent It can be shown \cite{DMN} that $\mathcal{E}^1(\gamma)|_t$ has second order analytic contact with $\gamma$ at $\gamma(t)$. Moreover, $\gamma(t)$ is a conformal vertex if and only if the order of contact is strictly bigger than 2. This highlights the fact that the conformal strain measures the infinitesimal distorsion of the curve from its osculating cycle.

\begin{defn}{If $\gamma$ is non-degenerate then there exist a unique null lift, referred to as the {\it canonical lift} such that $\mathcal{Q}_{\gamma}= \langle \Gamma'',\Gamma\rangle^2dt^4$. The smooth positive function $\upsilon_{\gamma}=|\langle \Gamma',\Gamma'\rangle|^{1/2}$ is the {\it conformal strain density} and the exterior differential $1$-form $\sigma_{\gamma}=\upsilon_{\gamma}dt$ is called the {\it conformal arc element} of $\gamma$. By construction, $\sigma_{\gamma}$  is invariant under the action of the restricted conformal group and orientation-preserving changes of parameter. In particular, each generic time-like curve can be parameterized is such a way that $\upsilon_{\gamma}=1$.  In this case, we say that the curve is {\it parameterized by conformal parameter}, which is usually denoted by $u$. Given a smooth map $f:I\to \R^h$, we define the {\it derivative of $f$ with respect to the conformal arc element} by $\dot{f}=\upsilon_{\gamma}^{-1}f'$.}\end{defn}

\begin{remark}
{The $1$-form $\sigma_{\gamma}$ is the Lorentzian analogue of the conformal arc element of a curve in the 3-dimensional round sphere \cite{LO2010,MMR,M1,MN} and generalizes the analogue notion for generic time-like curves in the $3$-dimensional Einstein universe \cite{DMN}.}
\end{remark}

\noindent Let $\gamma$ be non-degenerate and $\Gamma$ be its canonical lift. Then, $\mathrm{dim}(\mathcal{T}^3(\gamma)|_t)=4$, for every $t$. Hence $\mathcal{T}^3(\gamma)$ is a vector bundle. Furthermore, the cross sections
\begin{equation}\label{cvf1}
M_0=\Gamma,\quad M_1=\frac{1}{|\langle \dot{\Gamma},\dot{\Gamma}\rangle|^{1/2}}\dot{\Gamma},\quad
M_2=\frac{pr_{(2)}{\dddot{(\Gamma)}}}{\langle pr_{(2)}(\dddot{\Gamma}),pr_{(2)}(\dddot{\Gamma})\rangle)^{1/2}},
\end{equation}
and
\begin{equation}\label{cvf2}
M_{n+1}=-\frac{1}{\langle \Gamma,\ddot{\Gamma}\rangle}\ddot{\Gamma}+\frac{\langle \dot{\Gamma},\ddot{\Gamma}\rangle}{\langle \Gamma,\ddot{\Gamma}\rangle}\dot{\Gamma}+\frac{1}{2}(
\frac{\langle \ddot{\Gamma},\ddot{\Gamma}\rangle}{\langle \Gamma,\ddot{\Gamma}\rangle^2}
- \frac{\langle \ddot{\Gamma},\dot{\Gamma}\rangle^2}{\langle \Gamma,\ddot{\Gamma}\rangle^2})\Gamma
\end{equation}
give rise to a {\it canonical trivialization} $(M_0,M_1,M_2,M_{n+1})$ of $\mathcal{T}^3(\gamma)$. The canonical trivialization satisfies
\begin{equation}\label{CT}
\begin{split}
& \langle M_0,M_0\rangle = \langle M_{n+1},M_{n+1}\rangle = \langle M_0,M_1\rangle = \langle M_0,M_2\rangle=0,\\
& \langle M_1,M_{n+1}\rangle=\langle M_2,M_{n+1}\rangle=\langle M_1,M_2\rangle=0,\\
&\langle M_0,M_{n+1}\rangle = \langle M_1,M_1\rangle = -\langle M_2,M_2\rangle=-1.
\end{split}
\end{equation}
In particular, $M_0|_t\wedge (M_1|_t+M_2|_t)$ is an isotropic bi-vector, for every $t\in I$. If we revert the orientation along the curve, this bi-vector change sign. Thus, each non-degenerate curve possesses an {\it intrinsic orientation} such that $M_0|_t\wedge (M_1|_t+M_2|_t)\in \mathcal{C}_{+}$, for every $t\in I$.
\medskip

\noindent From now on we implicitly assume that the time-like curves are non-degenerate and equipped with their intrinsic orientations. In addition, if $\gamma$ is such a curve we use the notation $\Gamma$ to denote its canonical null lifting.

\begin{defn}{The scalar product (\ref{sp}) induces a metric structure on the vector bundle $\mathcal{N}^3(\gamma)$. The {\it covariant derivative} of a cross section $V:I\to \mathcal{N}^3(\gamma)$ with respect to the conformal line-element is defined by
\begin{equation}\label{D}
D(V)=pr_{(3)}(\dot{V}):I\to \mathcal{N}^3(\gamma).\end{equation}}\end{defn}

\noindent The normal bundles $\mathcal{N}^2(\gamma)$ and $\mathcal{N}^3(\gamma)$ possess two {\it canonical cross sections}, denoted by $W_{\gamma}$ and $S_{\gamma}$ respectively. The cross section $S_{\gamma}$ is defined by
\begin{equation}\label{SS1} S_{\gamma}=pr_{(3)}(\dot{M}_2).\end{equation}
We set
\begin{equation}\label{cinv1}h_1=\langle \dot{M}_1,M_{n+1}\rangle, \quad h_2=\sqrt{\langle S_{\gamma},S_{\gamma}\rangle}.
\end{equation}
and we define $W_{\gamma}$ by
\begin{equation}\label{VD}W_{\gamma}=(\dot{h_1}-3h_2\dot{h_2})M_2-(h_2^2-2h_1)S_{\gamma}+D^2(S_{\gamma}).\end{equation}
Note that
\[\begin{split}
\mathcal{T}^4(\gamma)|_t&=\mathrm{span}(M_0|_t,M_1|_t,M_2|_t,S_{\gamma}|_t,B_{n+1}|_t),\\
\mathcal{T}^5(\gamma)|_t&=\mathrm{span}(M_0|_t,M_1|_t,M_2|_t,S_{\gamma}|_t,D(S_{\gamma})|_t,B_{n+1}|_t),
\end{split}
\]
for every $t\in I$.

\subsection{First and second-order frames}
A {\it first-order frame} along $\gamma$ is a smooth map
$$A=(A_0,\dots A_{n+1}):I\to \mathrm{A}^{\uparrow}_+(2,n)$$
such that $A_0$ is a null lift and $A'_0\in \mathrm{Span}(A_0,A_1)$. First-order frame fields do exist along any time-like curve. If $A$ is a first-order frame, then any other is given by
$$\widetilde{A}=A\cdot X(r,x,y,R),$$
where
$$r,x : I\to \R, x>0,\quad y:I\to \R^{n-1},\quad R:I\to \mathrm{SO}(n-1)$$
are smooth functions and
\begin{equation}\label{gauge}X(r,x,y,R)=\left(
               \begin{array}{cccc}
                 r &-x & ^ty\cdot R & \frac{^ty\cdot y-x^2}{2} \\
                 0 & 1 & 0 & x/r \\
                 0 & 0 & R & y/r \\
                 0 & 0 & 0 & r^{-1} \\
               \end{array}
             \right).\end{equation}

\begin{defn}{A first-order frame is said of the {\it second-order} if
$$A_0=M_0,\quad A_1=M_1,\quad A_2=M_2,\quad A_{n+1}=M_{n+1}.$$
Second-order frames do exist along any generic, time-like curve with its intrinsic orientation. Note that $(A_3,\dots A_n)$ is a trivialization of the third-order normal bundle of $\gamma$.}\end{defn}

\noindent If $A$ is a second order frame field, then
\begin{equation}\label{LS}A' = A\cdot \left(
              \begin{array}{ccccc}
                0 & -h_1& 1 & 0 & 0 \\
                1 & 0 & 0 & 0 & h_1 \\
                0 & 0 & 0 & -^ts & 1 \\
                0 & 0 & s & \phi & 0 \\
                0 & -1 & 0 & 0 & 0 \\
              \end{array}
            \right)\upsilon_{\gamma},\end{equation}
where $\phi =(\phi^j_i)_{i,j=3,\dots,n} :I\to \mathfrak{o}(n-2)$ is a smooth map, $h_1$ is as in (\ref{cinv1}) and $s=^t(s^3,\dots,s^n)$ is defined by
$$S_{\gamma}=\sum_{j=3}^{n}s^jA_j.$$
If $V=\sum_{j=3}^{n}V^jA_j$ is a cross-section of $\mathcal{N}^3(\gamma)$, then
$$DV=\sum_{j=3}^{n}(\dot{V}^j+\sum_{i=3}^{n}\phi^j_i V^i)A_j.$$
We then have
\begin{equation}\label{D12S}DS_{\gamma}= \sum_{j=3}^n s^j_{(1)}A_j,\quad D^2(S_{\gamma})= \sum_{j=3}^n s^j_{(2)}A_j,\end{equation}
and
\begin{equation}\label{VDL}
W_{\gamma}=(\dot{h}_1-3^ts\cdot \dot{s})A_2+\sum_{j=3}^{n}(\dot{s}^j_{(2)}+(2h_1-^ts\cdot s)s^j)A_j,
\end{equation}
where $s_{(1)},s_{(2)}:I\to \R^{n-3}$ are given by
$$s^j_{(1)} = \dot{s}^j+\sum_{i=3}^{n}\phi^j_i s^i,\quad
s^j_{(2)} = \dot{s}^j_{(1)}+\sum_{i=3}^{n}\phi^j_i s^i_{(1)}.$$

\subsection{The strain functional and conformal world-lines}
\noindent We have seen that we may construct, for a generic time-like curve, a canonical line-element $\sigma_{\gamma}$ on $I$. If $K\subset I$ is a closed interval in $I$, we can define the {\it total strain functional}
$$\mathcal{S}_K(\gamma)=\int_K \sigma_{\gamma}$$
on the space of smooth, generic, time-like immersions of $I$ into $\mathcal{E}^{1,n-1}$.
\begin{defn}{We say that $\gamma$ is a {\it conformal world-line} if for any closed interval $K\subset I$ and any smooth variation
$$\mathbf{g}:(t,\tau)\in I\times (-\epsilon,\epsilon)\to \gamma_{\tau}(t)\in \mathcal{E}^{1,n-1},$$
with $\gamma_0=\gamma$ and\footnote{$\mathrm{supp}(\mathbf{g})$ denotes the support of the variation, ie the closure of the set of all $t\in I$ such that $\gamma_{\tau}(t)\neq \gamma(t)$ for some $\tau\in (-\epsilon,\epsilon)$.} $\mathrm{supp}(\mathbf{g})
\subset K$, we have $$\frac{d}{d\tau}\left( \mathcal{S}_K(\gamma_\tau)\right)|_t=0.$$}\end{defn}

\section{The variational equations}\label{s2}
\noindent The purpose of this section is to prove the following

\begin{thm}\label{thmA}{A generic time-like curve equipped with its intrinsic orientation is a conformal world-line if and only if $W_{\gamma}$ vanishes identically.
}\end{thm}

\begin{proof}{First we prove that a generic time-like curve satisfying $W_{\gamma}=0$ is a conformal world line. Without loss of generality we assume that $\gamma$ is parameterized by conformal parameter. Keep in mind that $W_{\gamma}=0$ if and only if
\begin{equation}\label{EL1}
\dot{h}_1=3h_2\dot{h}_2,\quad D^2(S_{\gamma})=(h_2^2-2h_1)S_{\gamma}.
\end{equation}
Let fix a second-order frame $A:I\to \mathrm{A}^{\uparrow}_+(2,n)$ and
$$\mathbf{g}:(u,\tau)\in I\times (-\epsilon,\epsilon)\to \gamma_{\tau}(u)\in \mathcal{E}^{1,n-1}$$
be a compactly supported variation of $\gamma$ such that $\mathrm{supp}(\mathbf{g})\subseteq K$. Eventually shrinking the interval $(-\epsilon, \epsilon)$, we may assume that $\gamma_{\tau}$ is generic and equipped with its intrinsic orientation, for every $\tau$. Then, there is a differentiable map
$$\mathcal{A}:(u,\tau)\in I\times (-\epsilon,\epsilon)\to \mathcal{A}_{\tau}(u)\in \mathrm{A}^{\uparrow}_+(2,n)$$
such that $\mathcal{A}_0=A$ and $\mathcal{A}_{\tau}$ is a second-order frame along $\gamma_{\tau}$, for every $\tau\in (-\epsilon,\epsilon)$. We then have
$$\mathcal{A}^{-1}d\mathcal{A} = Qdu+\Lambda dt,$$
where
$$Q=\left(
                                 \begin{array}{ccccc}
                                   0 & -m & v & 0 & 0 \\
                                   v & 0 & 0 & 0 & m \\
                                   0 & 0 & 0 & -^tp & 0 \\
                                   0 & 0 & p & \psi & 0 \\
                                   0 & v & 0 & 0 & 0 \\
                                 \end{array}
                               \right),$$
and
$$\Lambda=\left(
                                 \begin{array}{ccccc}
                                   \lambda_0^0 & -\lambda_{n+1}^1 & \lambda^2_{n+1} & ^t\lambda_{n+1} & 0 \\
                                   \lambda^1_0 & 0 & \lambda^2_1 & ^t\lambda_1 & \lambda^1_{n+1} \\
                                   \lambda^2_0 & \lambda^2_1 & 0 & -^t\lambda_2 & \lambda^2_{n+1} \\
                                   \lambda_0 & \lambda_1 & \lambda_2 & \lambda & \lambda_{n+1} \\
                                   0 & -\lambda^1_0 & \lambda^2_0 & \lambda_0 & -\lambda^0_0 \\
                                 \end{array}
                               \right).
$$
The entries
\[
\begin{split}
& v,m,\lambda^0_0,\lambda^1_0,\lambda^2_0,\lambda_{n+1}^1,\lambda^2_{n+1}
:I\times (-\epsilon,\epsilon)\to \R,\\
& p, \lambda_0,\lambda_2,\lambda_2,\lambda_{n+1}:I\times (-\epsilon,\epsilon)\to \R^{n-2},\\
& \psi,\lambda :I\times (-\epsilon,\epsilon)\to \mathfrak{o}(n-2).
\end{split}
\]
of $Q$ and $\Lambda$ are smooth maps. Their restrictions to $I\cong I\times \{0\}$ are denoted with the same symbols with an over-bar. Note that $\mathrm{supp}(\overline{\Lambda})\subseteq K$. The cross section
$$V_{\mathbf{g}}=\sum_{a=2}^{n}\overline{\lambda}^a_0 A_a:I\to \mathcal{N}^2(\gamma)$$
does not depend on the choice of $A$ and $\mathcal{A}$. We call $V_{\mathbf{g}}$ the {\it infinitesimal variation} of $\mathbf{g}$. From $\dot{A}=A\cdot \overline{Q}$ and taking into account (\ref{LS}), we obtain
$\overline{\upsilon}=1$, $\overline{m}_1=h_1$, $\overline{p}=s$ and $\overline{\psi}=\phi$.
From the Maurer-Cartan equations we have
$$\partial_u\Lambda - \partial_tQ = \Lambda\cdot Q - Q\cdot \Lambda.$$
This implies
\[
\begin{split}
&(\partial_t v)|_I = \overline{\lambda}_0^0 + \dot{\overline{\lambda}}^{1}_{0},\\
&\overline{\lambda}_0^0 =\frac{1}{2}(\dot{\overline{\lambda}}^0_{2}-\dot{\overline{\lambda}}^1_{0}-
\overline{\lambda}_{n+1}\cdot s-h_1\overline{\lambda}^2_1),\\
&\dot{\overline{\lambda}}_{0}=\overline{\lambda}_1+\overline{\lambda}_0^2s-\phi \cdot \overline{\lambda}_0,\\
&\dot{\overline{\lambda}}_1=-h_1\overline{\lambda}_0-\overline{\lambda}_{n+1}-\overline{\lambda}^2_1s-\phi\cdot\overline{\lambda}_1,\\
& \dot{\overline{\lambda}}_{n+1}=h_1\overline{\lambda}_1+\overline{\lambda}_2-\overline{\lambda}^2_0s-\phi\cdot \overline{\lambda}_{n+1},\\
&\dot{\overline{\lambda}}_0^2=\overline{\lambda}^2_1-s\cdot \overline{\lambda}_0.
\end{split}
\]
From these equalities, integrating by parts and taking into account that $\mathrm{supp}(\overline{\lambda}^i_j)\subset K$, $i,j=0,\dots, n+1$,
we have
\[\begin{split}
&\frac{d}{dt}\left(\mathcal{S}_K(\gamma_t)\right)|_{t=0}=\int_K (\partial_t v|_I)du =
\int_K (\overline{\lambda}_0^0+\dot{\overline{\lambda}}^1_0)du=\int_K \overline{\lambda}_0^0du=\\
&= -\frac{1}{2}\int_K (s\cdot(-\dot{\overline{\lambda}}_1-
h_1\overline{\lambda}_0-\overline{\lambda}_1^2 s-\phi\cdot \overline{\lambda}_1 )+h_1(\dot{\overline{\lambda}}_0^2+s\cdot \overline{\lambda}_0))du=\\
&=\frac{1}{2}\int_K(-\dot{s}\cdot\overline{\lambda}_1+h_1s\cdot \overline{\lambda}_0+(s\cdot s)\overline{\lambda}^2_1+s\cdot \phi\cdot \overline{\lambda}_1+\dot{h}_1\overline{\lambda}^2_0+h_1s\cdot\overline{\lambda}_0)du=\\
&=-\frac{1}{2}\int_K(\dot{s}\cdot(\dot{\overline{\lambda}}_0+\overline{\lambda}_0^2s+\phi\cdot \overline{\lambda}_0))du+\frac{1}{2}\int_K(s\cdot \phi\cdot (\dot{\overline{\lambda}}_0+\overline{\lambda}^2_0 s+\phi\cdot \overline{\lambda}_0))du\\
&+\frac{1}{2}\int_K(k_1s\cdot \overline{\lambda}_0+(^ts\cdot s)(\dot{\overline{\lambda}}^2_0-s\cdot \overline{\lambda}_0)+\dot{h}_1\overline{\lambda}^2_0+h_1s\cdot \overline{\lambda}_0)du=\\
&=\frac{1}{2}\int_K(\overline{\lambda}^2_0(\dot{h}_1-3 ^ts\cdot \dot{s})+
\overline{\lambda}_0\cdot (\dot{s}_{(1)}+\phi\cdot s_{(1)}-((^ts\cdot s)-2h_1)s))du=\\
&=\frac{1}{2}\int_K(\overline{\lambda}^2_0(\dot{h}_1-3 ^ts\cdot \dot{s})+
\overline{\lambda}_0\cdot (s_{(2)}+((2h_1-^ts\cdot s))s))du= \frac{1}{2}\int_K \langle V_{\mathbf{g}},W_{\gamma}\rangle du.
\end{split}\]
This proves the result.
\medskip

\noindent Next we show that for each $u_0\in I$ there exist an open interval $J\subset I$ containing $u_0$ such that for every smooth function $\rho:I\to \R$ with compact support $K\subset J$ and every $j=2,\dots ,n$, there exist a compactly supported variation $\mathbf{g}$ such that  $V_{\mathbf{g}}=\rho A_j$. This clearly implies that a conformal world-line satisfies $W_{\gamma}=0$.

\noindent Using the conformal invariance of the functional we may suppose, without loss of generality, that $\gamma(u)$ belongs to the Minkowski chamber, for every $u$ lying in an open interval $J\subset I$ containing $u_0$. Then, $\gamma|_J=\mathbf{j}\circ \alpha$ where $\alpha:J\to \mathbb{M}^{1,n-1}$ is a time-like curve of the Minkowski space. Let $\mathbf{t}:J\to \mathbb{M}^{1,n-1}$ be the future-directed time-like unit tangent vector along $\alpha$ and
$$N(\alpha)=\{(u,\mathbf{v})\in J\times \mathbb{M}^{1,n-1} : ^{*}\mathbf{x}\cdot \mathbf{t}|_u=0\}$$
be the normal bundle of $\alpha$, equipped with the metric covariant derivative\footnote{$pr|_u$ is the orthogonal projection of $\mathbb{M}^{1,n-1}$ onto $N(\alpha)|_u$} $\nabla(\mathbf{v})=pr(\dot{\mathbf{v}})$. Let $(\mathbf{b}_2,\dots,\mathbf{b}_n)$ be a flat orthogonal trivialization of $N(\alpha)$ such that $(\mathbf{t},\mathbf{b}_2,\dots,\mathbf{b}_n)$ is positive-oriented. Then $(\alpha,\mathbf{t},\mathbf{b}_2,\dots, \mathbf{b}_n):J\to P^{\uparrow}_+(1,n-1)$ is a lift of $\alpha$ to the restricted Poincar\'e group and\footnote{$\mathbf{J}$ is the faithful representation of $P^{\uparrow}_+(1,n)$ into the conformal group.}
$$F=\mathbf{J}\circ (\alpha,\mathbf{t},\mathbf{b}_2,\dots, \mathbf{b}_n):J \to \mathrm{A}^{\uparrow}_+(2,n)$$
is a first-order frame field along $\gamma|_J$. Let $A:J\to \mathrm{A}^{\uparrow}_+(2,n)$ be a second order frame along $\gamma$. We then have $A=F\cdot X(r,x,y,R)$, where
$$x,r:J\to \R, r>0,\quad y : J\to \R^{n-1},\quad R=(R^i_j)_{i,j=2,\dots n} :J\to \mathrm{SO(n-1)}$$
are smooth maps and $X(r,x,y,R)$ is as in (\ref{gauge}). We put
$$m=^t(m^2,\dots,m^n)=\frac{\rho}{r}^t(R^2_j,\dots, R^n_j).$$
It is now an easy matter to check that the variation
\[\begin{split}
\gamma_{\tau}(u) &= \mathbf{j}\circ(\alpha|_u+ \tau\sum_{j=2}^n m^j(u)\mathbf{b}_j|_u), \quad
(u,\tau)\in J\times (-\epsilon,\epsilon),\\
\gamma_{\tau}(u) &=\gamma(u),\quad  (u,\tau)\in (I-J)\times (-\epsilon,\epsilon),
\end{split}\]
satisfies the required properties. This concludes the proof of the Theorem.}\end{proof}

\noindent The next theorem shows that conformal world-lines lie in an Einstein universe of dimension 2,3 or 4.

\begin{thm}\label{ThmB}{The trajectory of a conformal world-line is contained in a $m$-dimensional Lorentzian cycle, with $2\le m \le 4$.}\end{thm}
\begin{proof}{Let $\gamma$ be a conformal world-line parameterized by conformal parameter. We consider a second-order frame field $A=(A_0,\dots A_{n+1})$ along $\gamma$. Then
$$\mathcal{T}^5(\gamma)|_u=\mathrm{Span}(A_0|_u,A_1|_u,A_2|_u,S_{\gamma}|_u,DS_{\gamma}|_u,A_{n+1}|_u),$$
for every $u\in I$. From (\ref{LS}) we have
\begin{equation}\label{STR1}\begin{split}
\dot{A}_0&=A_1,\quad \dot{A}_1=-h_1A_0-A_{n+1},\\
\dot{A}_2&=A_0+S_{\gamma},\quad \dot{A}_{n+1}=-h_1A_1+A_2,
\end{split}
\end{equation}
and
\begin{equation}\label{STR2}
\dot{A}_j=-s_jA_2+\sum_{k=3}^{n}\phi_j^kA_k,\quad  j=3,\dots, n.
\end{equation}
The first equation in (\ref{EL1}) implies
\begin{equation}\label{EL1Bis}(h_2)^2=\frac{2}{3}h_1+c_1,\quad c_1\in \R.\end{equation}
The remaining equations in (\ref{EL1}) can be written as
\begin{equation}
\label{EL2}
\dot{s}=-\phi\cdot s+s_{(1)},\quad \dot{s}_{(1)}=-(\frac{4}{3}h_1-c_1)s-\phi\cdot s_{(1)}
\end{equation}
or, equivalently, in the form
\begin{equation}
\label{EL3}\begin{cases}
\frac{d}{du}(S_{\gamma})=DS_{\gamma}-(\frac{2}{3}h_1+c_1)A_2,\\ \frac{d}{du}(DS|_{\gamma})=-(\frac{4}{3}h_1-c_1)S_{\gamma}-\frac{1}{3}\dot{h}_1A_2.
\end{cases}
\end{equation}
From (\ref{EL2}) we see that $s\wedge s_{(1)}$ is a solution of the linear system
$$\frac{d}{du}(s\wedge s_{(1)}) = -(\phi\cdot s)\wedge s_{(1)}- s\wedge (\phi\cdot s_{(1)}).$$
Hence, two possibilities may occur :
\begin{itemize}
\item Case I : $s|_u\wedge s_{(1)}|_u\neq 0$, for every $u\in I$
\item Case II : $s\wedge s_{(1)}$ vanishes identically.
\end{itemize}
{\it Case I}. The vectors fields $S_{\gamma}$ and $DS_{\gamma}$ are everywhere linearly independent. Then, the osculating spaces $\mathcal{T}^5(\gamma)|_u$ are six-dimensional.
Using (\ref{STR1}) and (\ref{EL3}) we obtain
$$\frac{d}{du}(A_0\wedge A_1\wedge A_2\wedge S_{\gamma}\wedge DS_{\gamma}\wedge A)=0.$$
Hence, $\mathcal{T}^5(\gamma)|_u$ coincides with a fixed $6$-dimensional Lorentzian subspace $\mathbb{W}_{\gamma}\subseteq \R^{2,n}$, for each $u\in I$. This implies that the trajectory of $\gamma$ is contained in the $4$-dimensional Lorentzian cycle $\mathcal{E}(\mathbb{W}_{\gamma})$.
\medskip

\noindent {\it Case II}. Since $(s,s_{(1)})$ is a solution of the linear system (\ref{EL2}), there are two possibilities :
\begin{itemize}
\item Case II.1 : $s=s_{(1)}=0$;
\item Case II.2 : $s^t\cdot s+^ts_{(1)}\cdot s_{(1)}>0$.
\end{itemize}
\noindent {\it Case II.1}. If $s=\dot{s}=0$, then (\ref{STR1}) implies
$$\frac{d}{du}(A_0\wedge A_1\wedge A_2\wedge A_{n+1})=0.$$
Therefore, $\mathcal{T}^3(\gamma)|_u$ coincides with a fixed $4$-dimensional Lorentzian subspace
$\mathbb{W}_{\gamma}\subseteq \R^{2,n}$, for every $u\in I$. Hence, the trajectory of $\gamma$ is contained in the $2$-dimensional Lorentzain cycle $\mathcal{E}(\mathbb{W}_{\gamma})$.

\noindent {\it Case II.2 } If $s^t\cdot s+^ts_{(1)}\cdot s_{(1)}>0$ and $s\wedge s_{(1)}=0$, then $\mathrm{Span}(S_{\gamma}|_u,DS_{\gamma}|_u)$ is a $1$-dimensional space-like subspace, for every $u\in I$. Thus
$$\mathcal{P}=\{(u,V)\in I\times \R^{2,n}: V\in \mathrm{Span}(S_{\gamma}|_u,DS_{\gamma}|_u)\}$$
is a rank $1$ vector bundle. Let $P$ be a unit-length cross section of $\mathcal{P}$. The identity $s\wedge s_{(1)}=0$ implies that $S_{\gamma}$ and $DS_{\gamma}$ are both proportional to $P$. We put $\widetilde{I}=I-I_*$, where $I_*$ is the discrete set $\{u\in I : S_{\gamma}|_u=0\}$. On $\widetilde{I}$ the fifth and fourth-order osculating bundles are spanned by $A_0,A_1,A_2,P,A_{n+1}$. This implies the existence of a smooth function $f:\widetilde{I}\to \R$ such that
$$\dot{P}|_{\widetilde{I}}=fP|_{\widetilde{I}},\quad \mathrm{mod}(A_0,A_1,A_2,A_{n+1}),$$
From (\ref{STR1}) and taking into account the previous identity we get
\begin{equation}\label{CC}\frac{d}{du}(A_0\wedge A_1\wedge A_2\wedge P\wedge A_{n+1})|_{\widetilde{I}}= f( A_0\wedge A_1\wedge A_2\wedge P\wedge A_{n+1})|_{\widetilde{I}}.\end{equation}
Let $\mathrm{Gr}_{5}(\R^{n+2})$ be the Grassmannian of the $5$-dimensional vector subspace of $\R^{n+2}$. From (\ref{CC}) it follows that the map
$$u\in I\to \mathrm{Span}(A_0|_u\wedge A_1|_u\wedge A_2|_u\wedge P|_u\wedge A_{n+1}|_u)\in \mathrm{Gr}_{5}(\R^{n+2})$$
is constant on $\widetilde{I}$. By continuity is constant on $I$. Then, the fourth-order osculating spaces coincide with a fixed $5$-dimensional Lorentzian subspace $\mathbb{W}_{\gamma}\subseteq \R^{2,n}$. Hence, the trajectory of $\gamma$ is contained in the $3$-dimensional Lorentzain cycle $\mathcal{E}(\mathbb{W}_{\gamma})$.}\end{proof}

\noindent
This theorem shows that if $\gamma$ is a conformal world-line then three possibilities may occur :
\begin{itemize}
\item $\gamma$ is trapped in a $4$-dimensional Einstein universe and there are no 3-dimensional Lorentzian cycles containing the trajectory of $\gamma$. If this occurs, we say that $\gamma$ is a {\it linearly full conformal world-line}.
\item the trajectory is trapped in a $3$-dimensional Lorentzian cycle but does not lie in any $2$-dimensional Lorentzian cycle.
\item the trajectory is trapped in a $2$-dimensional Lorentzian cycle.
\end{itemize}

\section{Linearly full conformal world-lines}\label{s3}

\noindent From now on we limit ourselves to consider linearly full conformal world-lines of the four-dimensional Einstein universe. We also suppose that the parametrization is by the conformal parameter and that the orientation is the intrinsic one. These assumptions will be implicitly assumed throughout the subsequent discussions.

\subsection{Canonical frames and conformal curvatures}\label{ss:conf-fr}
\begin{prop}\label{propuniq}{Let $\gamma : I\to \mathcal{E}^{1,3}$ be a conformal world-line. Then, there exist a unique second-order conformal frame $B:I\to A^{\uparrow}_+(2,3)$ such that
\begin{equation}\label{mcc}\dot{B}=B\cdot \mathcal{K},\end{equation}
where
\begin{equation}\label{mcc2}\mathcal{K}=\left(
                             \begin{array}{cccccc}
                               0 & -k_1 & 1 & 0 & 0 & 0 \\
                               1 & 0 & 0 & 0 & 0 & k_1 \\
                               0 & 0 & 0 & -k_2 & 0 & 1 \\
                               0 & 0 & k_2 & 0 & -k_3 & 0 \\
                               0 & 0 & 0 & k_3 & 0 & 0 \\
                               0 & -1 & 0 & 0 & 0 & 0 \\
                             \end{array}
                           \right).
\end{equation}
 and $k_1,k_2,k_3:I\to \R$ are smooth functions with $k_2>0$ and $k_3\neq 0$. In addition, if $k_1$, $k_2$ and $k_3$ are non-constant then there exist $c_1,c_2,c_3\in \R$ with $c_3\neq 0$ such that
\begin{equation}\label{eqmrk4}
k_1=\frac{3}{2}k_2^2+c_1,\quad k_3=c_3k_2^{-2},\quad \dot{k}_2^2+k_2^4+c_3^2k_{2}^{-2}+2c_1k_2^2+c_2=0.
\end{equation}}\end{prop}
\begin{proof}{Let $(M_0,M_1,M_2,M_{5})$ be the canonical trivialization of $\mathcal{T}^3(\gamma)$ and $S_{\gamma}:I\to \mathcal{N}^3(\gamma)$ be the cross section defined as in (\ref{SS1}). We set $$B_3=S_{\gamma}/\sqrt{\langle S_{\gamma},S_{\gamma}\rangle}$$
and we pick the unique unit cross section $B_4:I\to \mathcal{N}^4(\gamma)$ such that $\mathrm{det}(M_0,M_1,M_2,M_3,B_4,M_5)>0$. Putting
$$B_0=M_0,\quad B_1=M_1,\quad B_2=M_2,\quad B_5=M_5,$$
the map $B=(B_0,\dots, B_5)$ is a second-order frame along $\gamma$. In view of (\ref{LS}) we have
\begin{equation}\label{FR1}\dot{B}_0=B_1,\quad \dot{B}_1=-k_1B_0-B_5,\quad \dot{B}_5=k_1B_1+B_2\end{equation}
where\footnote{cfr (\ref{cinv1}) for the definition of $h_1$} $k_1=h_1$. Since $D$ is a metric covariant derivative, $DB_3$ is a multiple of $B_4$ and hence $DB_3=k_3B_4$, for some non-zero function $k_3$. From (\ref{LS}) we have
\begin{equation}\label{FR2}\dot{B}_2=B_0+k_2B_3,\end{equation}
where $k_2=\sqrt{\langle S_{\gamma},S_{\gamma}\rangle}>0$. Keeping in mind (\ref{EL3}) we get
\begin{equation}\label{FR3}\dot{B}_3=-k_2B_2+k_3B_4.\end{equation}
Furthermore, since $B$ takes values in $\mathrm{A}^{\uparrow}_{+}(2,4)$ and using (\ref{FR1})-(\ref{FR3}) we obtain
\begin{equation}\label{FR4}\dot{B}_4=-k_3B_3.\end{equation}
Combining (\ref{FR1})-(\ref{FR4}) we deduce that $B$ satisfies (\ref{mcc}).

\noindent Suppose that the functions $k_1,k_2$ and $k_3$ are non-constant. The first and second covariant derivatives of $S_{\gamma}$ can be written as
$$DS_{\gamma}=\dot{k}_2B_3+k_2k_3B_4,\quad D^2S_{\gamma}=(\ddot{k}_2-k_2k_3^2)B_3+(2\dot{k}_2k_3+k_3\dot{k}_2)B_4.$$
Then, from variational equation $W_{\gamma}=0$ we get
\begin{equation}\label{EM2}
\ddot{k}_2=k_2(k_3^2+k_2^2-2k_1),\quad k_2\dot{k}_3+2k_3\dot{k}_2=0\quad \dot{k}_1=3k_2\dot{k}_2.
\end{equation}
The third and the second equations in (\ref{EM2}) imply the existence of two constants $c_1$ and $c_3\neq 0$ such that
\begin{equation}\label{EM3}k_1=\frac{3}{2}k_2^2+c_1,\quad k_3=c_3k_2^{-2}.\end{equation}
Substituting (\ref{EM3}) into (\ref{EM2}) we find
$$\ddot{k}_2=-2k_2^{3}+c_3^2k^{-3}-4k_2-2c_1k_2.$$
Then there exist a constant $c_2$ such that
$$\dot{k}_2^2+k_2^4+c_3^2k_{2}^{-2}+2c_1k_2^2+c_2=0.$$}\end{proof}

\begin{defn}{We call $B$ the {\it canonical conformal frame}. The functions $k_1,k_2,k_3$ are said the {\it conformal curvatures}. The constants $c_1,c_2,c_3$ are referred to as the {\it characters} of the world-line.}\end{defn}

\begin{remark}\label{remarkuniq}{Conversely, if $k_1,k_2,k_3:I\to \R$ are smooth functions satisfying (\ref{eqmrk4}), by solving the linear system (\ref{mcc}) with initial condition $\mathrm{B}(u_0)\in \mathrm{A}^{\uparrow}_+(2,4)$, we get a smooth map $\mathrm{B}:I\to \mathrm{A}^{\uparrow}_+(2,4)$. Then, $\gamma=[B_0]$ is a linearly full conformal world-line with conformal curvatures $k_1,k_2,k_3$ and canonical conformal frame $B$. Any other world-line $\widetilde{\gamma}$ with the same curvatures is congruent to $\gamma$ with respect to the restricted conformal group, ie there exist $\mathbf{X}\in \mathrm{A}^{\uparrow}_+(2,4)$ such that $\widetilde{\gamma}=\mathbf{X}\cdot \gamma$.}\end{remark}

\begin{remark}{The sign ambiguity of the third curvature can be removed with the following reasoning :  if the third conformal curvature of $\gamma$ is $k_3$ and if $\mathbf{X}$ is an orientation-preserving and time-reversing conformal transformation, then $\mathbf{X}\cdot \gamma(-u)$ is a conformal world-line with conformal curvatures $k_1(-u),k_2(-u)$ and $-k_3(-u)$. Therefore, up to a time-reversing conformal transformation, the curvature $k_3$ is positive.}\end{remark}

\begin{defn}{The sign of the third curvature is said the {\it conformal helicity} of the world-line. By the previous remark we see that is not restrictive to consider conformal world-lines with positive helicity. From now on this additional assumption is implicitly assumed.}\end{defn}

\begin{remark}{If $k_1,k_2$ and $k_3$ are constant, then the variational equations imply
$k_1=(k_2^2+k_3^2)/2$. In this case $\mathcal{K}$ is a fixed element of the Lie algebra $\mathfrak{a}(2,4)$ and $\gamma$ is congruent to the orbit through $[^t(1,0,0,0,0)]$ of the $1$-parameter group of conformal transformations $u\in \R\to \mathrm{Exp}(u\mathcal{K})\in A^{\uparrow}_+(2,4)$.
Therefore, from a theoretical point of view the determination of the conformal world-lines with constant curvatures is reduced to the calculation of the exponentials of the matrices $u\mathcal{K}$. From a computational point of view this is an elementary but not completely straightforward matter and
requires a detailed analysis of the possible orbit types of the infinitesimal generator
$\mathcal{K}$.}\end{remark}

\subsection{Conformal curvatures in terms of Jacobi's elliptic functions} From now on we suppose that the conformal curvatures are non-constant.

\begin{lemma}{If $c_1,c_2,c_3$ are the characters of $\gamma$ then the third-order polynomial $Q_1(t)=t^3+2c_1t^2+c_2t+c_3^2$ has three distinct real roots $e_1,e_2,e_3$ such that $e_1<0<e_2<e_3$.}\end{lemma}
\begin{proof}{Let $e_1,e_2,e_3$ be the roots of $Q_1$. Then,
$$\dot{k}_2^2+k_2^4+c_3^2k_{2}^{-2}+2c_1k_2^2+c_2=0$$
implies
\begin{equation}
\label{eee1}
(k_2\dot{k}_2)^2=-(k_2^2-e_1)(k_2^2-e_2)(k_2^2-e_3)
\end{equation}
and
\begin{equation}\label{eee2}
c_1=-\frac{1}{2}(e_1+e_2+e_3),\quad c_2=e_1e_2+e_1e_3+e_2e_3,\quad c_3^2=-e_1e_2e_3.\end{equation}
If two roots are complex conjugate each other, say $e_2$ and $e_3$, then the third equation of (\ref{eee2}) implies $e_1<0$. Hence the right hand side of (\ref{eee1}) is strictly negative. This contradicts the fact that $k_2$ is non constant. If $Q_1$ has a double root, say $e_2=e_3$, then the third equation of (\ref{eee2}) implies $e_1<0$. So, as in the previous case, the right hand side of (\ref{eee1}) is strictly negative. But this can't occur. Therefore, the roots are real and distinct. We choose the ordering $e_1<e_2<e_3$. Using again the third equation of (\ref{eee2}) we see that two possibilities may occur : either $e_1<e_2<e_3<0$ or else $e_1<0<e_2<e_3$. In the first case the right hand side of (\ref{eee1}) is negative. But, as in the previous cases, this conclusion is impossible. Thus the Lemma is proved.}\end{proof}

\begin{defn}{We say that $e_1,e_2,e_3$ are the {\it phase parameters} of the world-line. Note that, according to (\ref{eee2}),  the characters can be read off from the phase parameters.}\end{defn}

\noindent We put
\begin{equation}\label{ell}\ell_1=e_2\ell_3,\quad \ell_2=e_1\ell_4,\quad \ell_3=e_3-e_1,\quad \ell_4=e_3-e_2,\quad m=\ell_4/\ell_3,\end{equation}
and we denote by $K(m)$ and $\mathrm{sn}(-,m)$ the complete integral of the first kind and the
Jacobi's $sn$-function with parameter $m$\footnote{The parameter $m$ is the square of the modulus $k$ of the elliptic function. The reader pay attention to the fact that in the literature is also used the notation
$\mathrm{sn}(-,k)$ to denote the $\mathrm{sn}$-function with modulus $k$.}.

\begin{prop}\label{ConfCurv}{Let $\gamma$ be a conformal world-line with phase parameters $e_1<0<e_2<e_3$, then
\begin{equation}\label{k21}k_2(u)=\sqrt{\frac{\ell_1-\ell_2\mathrm{sn}^2(\sqrt{\ell_3}u+u_0,m)}
{\ell_3-\ell_4\mathrm{sn}^2(\sqrt{\ell_3}u+u_0,m)}},
\end{equation}
and
\begin{equation}\label{k13}k_1=\frac{3}{2}k_2^2-\frac{1}{2}(e_1+e_2+e_3),\quad k_3= \sqrt{-e_1e_2e_3}k_2^{-2},\end{equation}
where $u_0$ is a constant.
}\end{prop}
\begin{proof}{Put $f(u)=k_2(u/\sqrt{\ell_3})$ then, using (\ref{eee1}), we get
$$\dot{f}^2=-\frac{4}{e_3-e_1}(f-e_1)(f-e_2)(f-e_3).$$
If $e_1<0<e_2<e_3$, the general solution of the equation above (cfr. \cite{BF} pag. $77$) is
$$f(u)=\frac{\ell_1-\ell_2\mathrm{sn}^2(u+u_0,m)}{\ell_3-\ell_4\mathrm{sn}^2(u+u_0,m)},$$
where $u_0$ is a constant. This implies (\ref{k21}). We conclude the proof by observing that (\ref{k13}) is an immediate consequence of (\ref{EM3}) and (\ref{eee2}).
}\end{proof}

\noindent This proposition has two consequences : $k_2$ is a strictly positive, even periodic function with period
\begin{equation}\label{period}\omega=2K(m)/\sqrt{\ell_3}\end{equation}
and the parameterizations by the conformal parameter of a world-line are defined on the whole real line. In addition, with a shift of the independent variable the constant $u_0$ in (\ref{k21}) can be put equal to zero. This means that a conformal world line admits a {\it canonical parametrization} by conformal parameter such that $k_2(0)=\sqrt{\ell_1/\ell_3}$.

\begin{defn}{We say that $\gamma$ is the {\it standard configuration} of a world-line if $\gamma$ is parameterized by the canonical parameter and if, in addition, its canonical frame satisfies the initial condition $B|_0=\mathrm{Id}|_{6\times 6}$. Clearly, each world-line is conformally equivalent to a unique standard configuration.}\end{defn}

\noindent To summarize what has been said, we state the following Proposition :

\begin{prop}\label{PHRS}{The standard configurations of world-lines are in one-to-one correspondence with the point of the domain
$$\mathcal{F}=\{\mathbf{e}=(e_1,e_2,e_3)\in \R^3 : e_1<0<e_2<e_3\}\subset \R^3.$$
In other words, for each $\mathbf{e}=(e_1,e_2,e_3)\in \mathcal{F}$ there exist a unique world-line $\gamma$ in its standard configuration with phase parameters $(e_1,e_3,e_3)$.}\end{prop}

\subsection{The momentum operator} Let $\gamma$ be a standard configuration of a conformal world-line with curvatures $k_1,k_2,k_3$. We denote by $\mathcal{H}: Y\to \mathfrak{a}(2,4)$ the map
\begin{equation}\label{obs}\mathcal{H}=
\left(
\begin{array}{cccccc}
0 & -1 & -(k_2^2-k_1) & \dot{k}_2 & k_2k_3 & 0 \\
0 & 0 & 0 & -k_2 & 0 & 1 \\
-1 & 0 & 0 & 0 & 0 & -(k_2^2-k_1) \\
0 & -k_2 & 0 & 0 & 0 & \dot{k}_2 \\
0 & 0 & 0 & 0 & 0 & k_2k_3 \\
0 & 0 & -1 & 0 & 0 & 0 \\
\end{array}
\right).\end{equation}
Then, (\ref{k21}) and (\ref{k13}) imply
\begin{equation}\label{lax} \dot{\mathcal{H}}=[\mathcal{H},\mathcal{K}].\end{equation}
This equation together with $B'=B\cdot \mathcal{K}$ implies
\begin{equation}\label{momentum}B|_u\cdot \mathcal{H}|_u\cdot B^{-1}|_u = \mathcal{H}|_0,\quad \forall u\in \R. \end{equation}

\begin{defn}{We put $\mathfrak{m}=\mathcal{H}|_0$ and we say that $\mathfrak{m}$ is the  {\it momentum operator} of $\gamma$}\end{defn}

\begin{remark}\label{ThEx}{We give a brief explanation of the conceptual origin of the momentum operator. The first step is the construction of the {\it momentum space} and of the {\it Euler-Lagrange exterior differential system} \cite{Gr, GM}. In our specific situation, the momentum space is the $19$-dimensional manifold $Z=A^{\uparrow}_+(2,4)\times \mathfrak{K}$, where $$\mathfrak{K}=\{\mathfrak{k}=(\mathfrak{k}_1,\mathfrak{k}_2,\mathfrak{k}_3,\dot{\mathfrak{k}}_2)\in \R^4, \mathfrak{k}_2>0, \mathfrak{k}_3>0\}\subset \R^4.$$
The restricted conformal group acts freely on the left of $Z$ by
$$L_{Y}(X,\mathfrak{k})= (Y\cdot X,\mathfrak{k}),\quad \forall X,Y\in A^{\uparrow}_+(2,4), \forall \mathfrak{k}\in \mathfrak{K}.$$
The Euler-Lagrange differential system is the $A^{\uparrow}_+(2,4)$-invariant Pfaffian differential ideal $\mathcal{I}\subset \Omega^*(Z)$ generated by the $1$-forms
$$\mu^2_0,\quad \mu^3_0,\quad \mu^4_0,\quad \mu^2_1,\quad \mu^3_1,\quad \mu^4_1,\quad \mu^3_5,\quad \mu^4_5,\quad \mu^4_2,\quad \mu^0_0
$$
$$\mu^2_5-\mu^1_0,\quad \mu^1_5-\mathfrak{k}_1\mu^1_0,\quad \mu^3_2-\mathfrak{k}_2\mu^1_0,\quad
\mu^4_3-\mathfrak{k}_3\mu^1_0$$
and
$$d\dot{\mathfrak{k}}_2-\mathfrak{k}_2(\mathfrak{k}_3^2+\mathfrak{k}_2^2-2\mathfrak{k}_1)\mu^1,\quad
d\mathfrak{k}_2-\dot{\mathfrak{k}}_2\mu^1_0,\quad d\mathfrak{k}_1-3\mathfrak{k}_2\dot{k}_2\mu^1_0,\quad
\mathfrak{k}_2d\mathfrak{k}_3+2\mathfrak{k}_2\dot{\mathfrak{k}}_2\mu^1_0.
$$
The independence condition of the system is the invariant $1$-form $\mu^1_0$. The integral curves of $(\mathcal{I},\mu^1_0)$ can be build as follows : let $\gamma$ be
a conformal world-line\footnote{Remember that $\gamma$ is linearly full, with positive helicity and parameterized by the conformal parameter}  with curvatures $k_1,k_2,k_3$ and canonical frame $B$, the lifting $\mathfrak{b}=(B,k_1,k_2,k_3,\dot{k}_2):\R\to Z$ of $\gamma$ to $Z$ is said the {\it extended frame} along $\gamma$. Proposition \ref{propuniq} tell us that the integral curves of $(\mathcal{I},\mu^1_0)$ are the extended frames of the world-lines. Using the Maurer-Cartan equations (\ref{MC}) we see that
\begin{equation}\label{Liouville}\zeta = \frac{1}{2}(\mu^1_0+\mu^2_5+(\mathfrak{k}_2^2-\mathfrak{k}_1)\mu^2_0-\dot{\mathfrak{k}}_2\mu^3_0-
\mathfrak{k}_2\mathfrak{k}_3\mu^4_0+\mathfrak{k}_2\mu^3_1)
\end{equation}
is an invariant contact $1$-form such that the integral curves of its characteristic vector field\footnote{ie the vector field defined by $\zeta(X_{\zeta})=1$, $\iota_{X_{\zeta}}d\zeta=0$.} $X_{\zeta}$ are the extended frames of the world-lines. We can think of $\zeta$ as a map into the dual space $\mathfrak{a}(2,4)^*$ of the conformal Lie-algebra. Denote by $\natural:\mathfrak{a}^*(4,2)\to \mathfrak{a}(2,4)$ the pairing defined by the Killing form, then $\zeta^{\natural}=\natural \circ \zeta : Z\to \mathfrak{a}(2,4)$ is an equivariant map that, in our context, plays the role of the Legendre transformation. If $\mathfrak{b}$ is the extended frame of $\gamma$, then $\mathcal{H}$ coincides with $\zeta^{\natural}\circ \mathfrak{b}$. The momentum map of action of $A^{\uparrow}_+(2,3)$ on $(Z,\zeta)$ is given by
$$\widehat{\mathfrak{m}}:(X,\mathfrak{k})\in Z\to ad^*_X(\zeta|_{(B,\mathfrak{k})})\in \mathfrak{a}^*(2,4).$$
The {\it N\"other conservations theorem} for a Hamiltonian action on a contact manifold says that $\widehat{\mathfrak{m}}$ is constant along the characteristic curves. Thus, if $\mathfrak{b}$ is the extended frame of a world-line $\gamma$ then $\widehat{\mathfrak{m}}_{\gamma}=\widehat{\mathfrak{m}}\circ \mathfrak{b}$ is constant and  $\widehat{\mathfrak{m}}_{\gamma}^{\natural}$ is the momentum operator of $\gamma$. This explains the geometrical origin of the momentum operator of a world-line.}\end{remark}

\noindent Hereafter we will adopt the following notations:
\begin{itemize}
\item $P_{\mathfrak{m}}$ is the characteristic polynomial of $\mathfrak{m}$;
\item $\mathcal{S}_{\mathfrak{m}}$ is the spectrum of $\mathfrak{m}$, viewed as an endomorphism of $\C^6$;
\item for each $\lambda\in \mathcal{S}_{\mathfrak{m}}$, $n_1(\lambda)$ is the multiplicity of $\lambda$ as root of  $P_\mathfrak{m}$ and $n_2(\lambda)$ is the complex dimension of the $\mathfrak{m}$-eigenspace $\mathbb{V}_{\lambda}$ of $\lambda$.
\end{itemize}

\begin{defn}{A conformal world-line is said to be {\it regular, exceptional or singular} depending on whether $\mathfrak{m}$ is a regular, exceptional or singular element of the Lie algebra $\mathfrak{a}(2,4)$. In other words, $\gamma$ is regular if $\mathcal{S}_{\mathfrak{m}}$ consists of six elements, is exceptional if $\mathcal{S}_{\mathfrak{m}}$ has less that six elements and $n_2(\lambda)=1$, for every $\lambda \in \mathcal{S}_{\mathfrak{m}}$ and is singular if $n_2(\lambda)>1$, for some $\lambda$.}\end{defn}

\noindent  From (\ref{momentum}) it follows that $P_\mathfrak{m}$ coincides with the characteristic polynomial
of $\mathcal{H}|_u$, for every $u\in \R$. Keeping in mind (\ref{eqmrk4}) we get
\begin{equation}\label{chpolm}P_\mathfrak{m}(t)=t^6+2c_1t^4+(c_2+1)t^2+c_3^2 = Q_1(t^2)+t^2.\end{equation}
Let $Q_2(t)$ be the third-order polynomial $Q_1(t)+t^2$. Since the roots of $Q_1(t)$ are the phase parameters $e_1<0<e_2<e_3$, we infer that three possibilities may occur :
\begin{itemize}
\item $Q_2$ has three distinct real roots $\rho_1,\rho_2,\rho_3$ such that
$e_1<\rho_1<0<e_2<\rho_2<\rho_3<e_3$.
\item $Q_2$ has one negative real root $\rho_1$ with $e_1<\rho_1$ and two complex conjugate roots $\rho_2=\mu+i\nu$, $\nu>0$ and $\rho_3=\mu-i\nu$.
\item $Q_2$ has one simple real root $\rho_1$ such that $e_1<\rho_1<0$ and a double real root $\rho_2=\rho_3$, $e_2<\rho<e_3$.
\end{itemize}
We put\footnote{If $z\in \C-\R^+$, then $\sqrt{z}$ is the determination of the square root with positive imaginary part.}
\begin{equation}\label{eigenvl}\begin{split}
\lambda_0 &=i\sqrt{|\rho_1|},\quad \lambda_1=-i\sqrt{|\rho_1|},\quad \lambda _2=\sqrt{\rho_2},\\ \lambda_3&=-\sqrt{\rho_2},\quad \lambda_4=\sqrt{\rho_3},\quad \lambda_5=-\sqrt{\rho_3}.
\end{split}\end{equation}
In the first two cases the eigenvalues of $P_\mathfrak{m}$ are simple and
\begin{equation}\label{sptr1}\mathcal{S}_{\mathfrak{m}}=\{\lambda_0,\dots, \lambda_5\}.\end{equation}
While, in the third case we have
\begin{equation}\label{sptr1}\mathcal{S}_{\mathfrak{m}}=\{\lambda_0,\dots, \lambda_3\}\end{equation}
and the two real eigenvalues $\lambda_2$ and $\lambda_3$ are the double roots of $P_{\mathfrak{m}}$.


\begin{prop}\label{CT}{A conformal world-line is either regular or exceptional.}\end{prop}

\begin{proof}{For every $\lambda\in \mathcal{S}_{\mathfrak{m}}$ we define
$L_{\lambda}=(L_{\lambda}^0,\dots L_{\lambda}^5):\R\to \C^6$
by
\begin{equation}\label{L}\begin{cases}
L^0_{\lambda} =\lambda(\lambda^2-k_2^2)(\lambda^2+k_1-k_2^2),\\
L^1_{\lambda}=\lambda(\lambda-k_2\dot{k}_2),\\
L_{\lambda}^2=-\lambda^2(\lambda^2-k_2^2),\\
L^3_{\lambda}=\lambda(\lambda \dot{k}_2-k_2),\\
L^4_{\lambda}=k_2k_3(\lambda^2-k_2^2),\\
L_{\lambda}^5=\lambda(\lambda^2-k_2^2).
\end{cases}\end{equation}

\noindent The map $L_{\lambda}$ is real-analytic and periodic, with period $\omega$. Let $D_{\lambda}$ be its zero set. If $\lambda\notin \R$ then $D_{\lambda}=\emptyset$ and, if $\lambda\in \R$, we have
\begin{equation}\label{zeroes}
D_{\lambda}=\begin{cases} D^+_{\lambda}=\{n\omega+p_{\lambda}, n\in \mathbb{Z}\},\quad \lambda >0,\\
D^-_{\lambda}=\{n\omega-p_{\lambda}, n\in \mathbb{Z}\},\quad \lambda <0,
\end{cases}
\end{equation}
where $p_{\lambda}\in (0,\omega)$ is given by
\begin{equation}\label{palpha}
p_{\lambda}=\mathrm{sn}^{-1}\left(\frac{1}{\alpha},m\right),\quad \alpha=\sqrt{\frac{\ell_2-\lambda^2\ell_4}{\ell_1-\lambda^2\ell_3}}.
\end{equation}
The $\lambda$-eigenspace $\mathcal{V}_{\lambda}|_u$ of $\mathcal{H}|_u$
is spanned by $L_{\lambda}|_u$, for every  $u\notin D_{\lambda}$. Since $D_{\lambda}$ is a discrete set, this implies that $\mathrm{dim}(\mathcal{V}_j|_u)=1$, for every $u\in \R$. On the other hand, $\mathfrak{m}$ and $\mathcal{H}|_u$ belongs to the same adjoint orbit, for every $u$. Thus, also the eigenspaces of the momentum are $1$-dimensional. This yields the result.}\end{proof}

\begin{figure}[ht]
\begin{center}
\includegraphics[height=5cm,width=6cm]{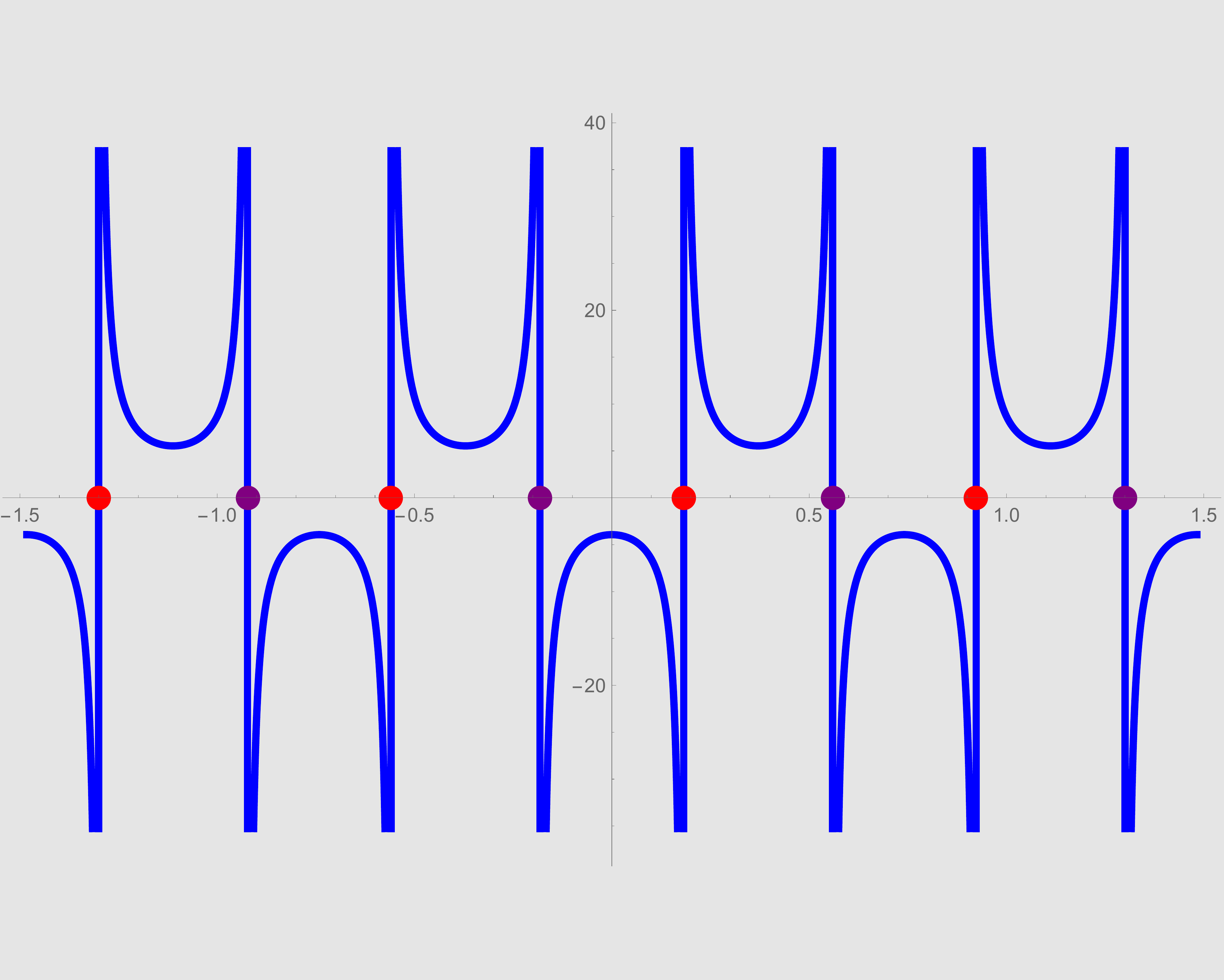}
\includegraphics[height=5cm,width=6cm]{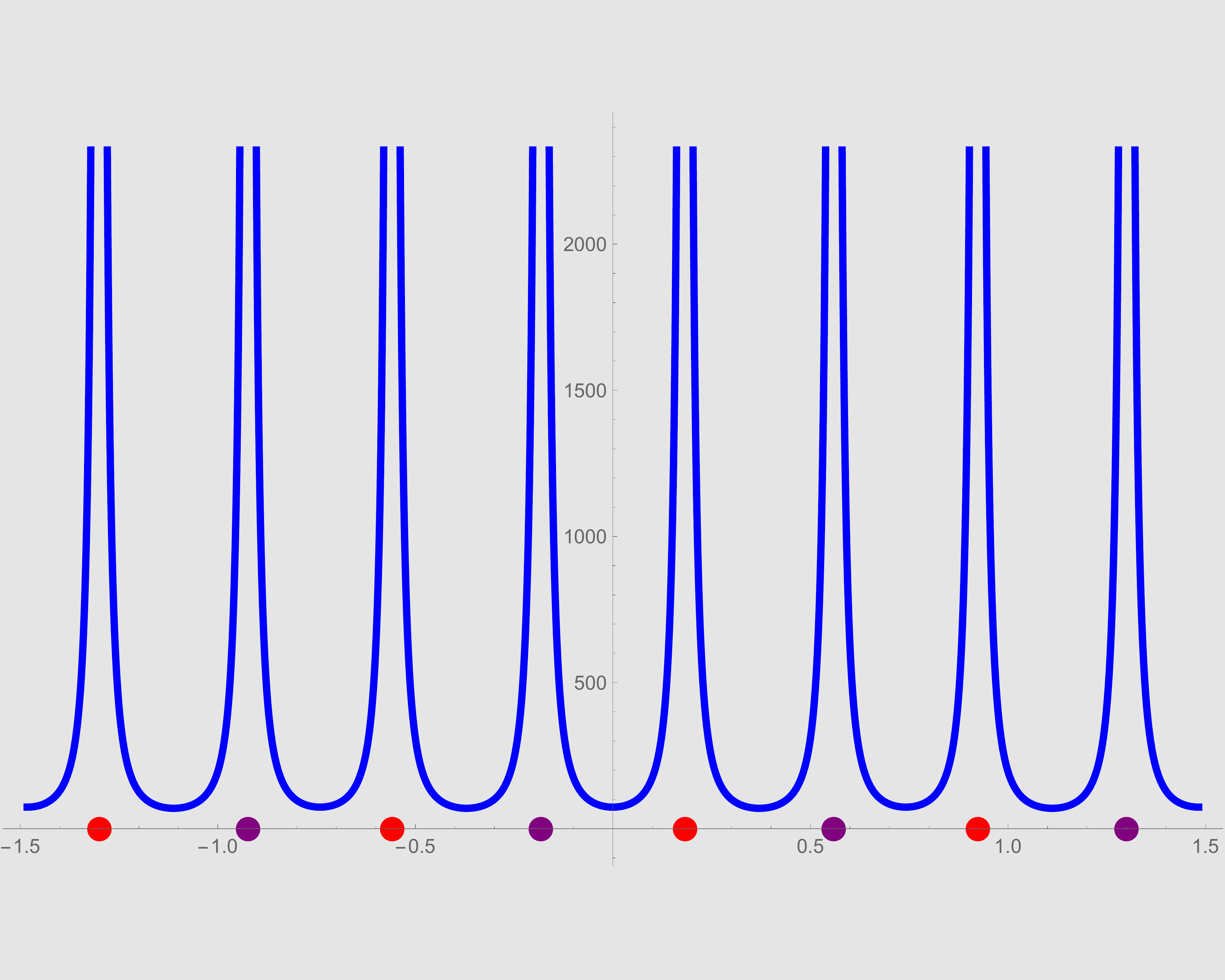}
\caption{The graphs of the functions $s_{\lambda}$ (on the left) and $r_{\lambda}$ (on the right), $\lambda\in \R$.}\label{FIG1}
\end{center}
\end{figure}

\subsection{Integrating factors and principal vectors} Let $\gamma$ be the standard configuration of a linearly full conformal string with non-constant curvatures. Denote by $e_1<0<e_2<e_3$ its phase parameters and by $B:\R\to A^{\uparrow}_+(2,4)$ its canonical frame. For each $\lambda\in \mathcal{S}_{\mathfrak{m}}$ we consider the functions
\begin{equation}\label{if2}r_{\lambda}=\frac{k_2\dot{k}_2+\lambda}{k_2^2-\lambda^2},\quad
s_{\lambda}=\frac{\lambda^2+k_2^2-2\lambda k_2\dot{k}_2}{(\lambda^2-k_2^2)^2}.
\end{equation}
If $\lambda\notin \R$, the functions $r_{\lambda}$ and $s_{\lambda}$ are periodic, complex-valued and real-analytic; if $\lambda\in \R$, $r_{\lambda}$ and $s_{\lambda}$ are periodic, real-valued and real-analytic on the complement of the discrete set
    $\widetilde{D}_{\lambda}=D^+_{\lambda}\cup D^-_{\lambda}$; their absolute values tend to infinity when $u$ approaches one of the points of $\widetilde{D}_{\lambda}$ (see Figure \ref{FIG1}). For notational consistency, we put $\widetilde{D}_{\lambda}=\emptyset$ when $\lambda\notin \R$.

\begin{defn}{A primitive $\delta_{\lambda}:\R-\widetilde{D}_{\lambda}\to \C$ of $r_{\lambda}$ is said an {\it integrating factor of the first kind} for the eigenvalue $\lambda$ if
\begin{itemize}
\item  $\delta_{\lambda}|_0=0$;
\item $e^{-\delta_{\lambda}}L_{\lambda}:\R-\widetilde{D}_{\lambda}\to \C^6$ extends to a real-analytic map $\R\to \C^6$.
\end{itemize}}\end{defn}

\begin{figure}[ht]
\begin{center}
\includegraphics[height=5cm,width=6cm]{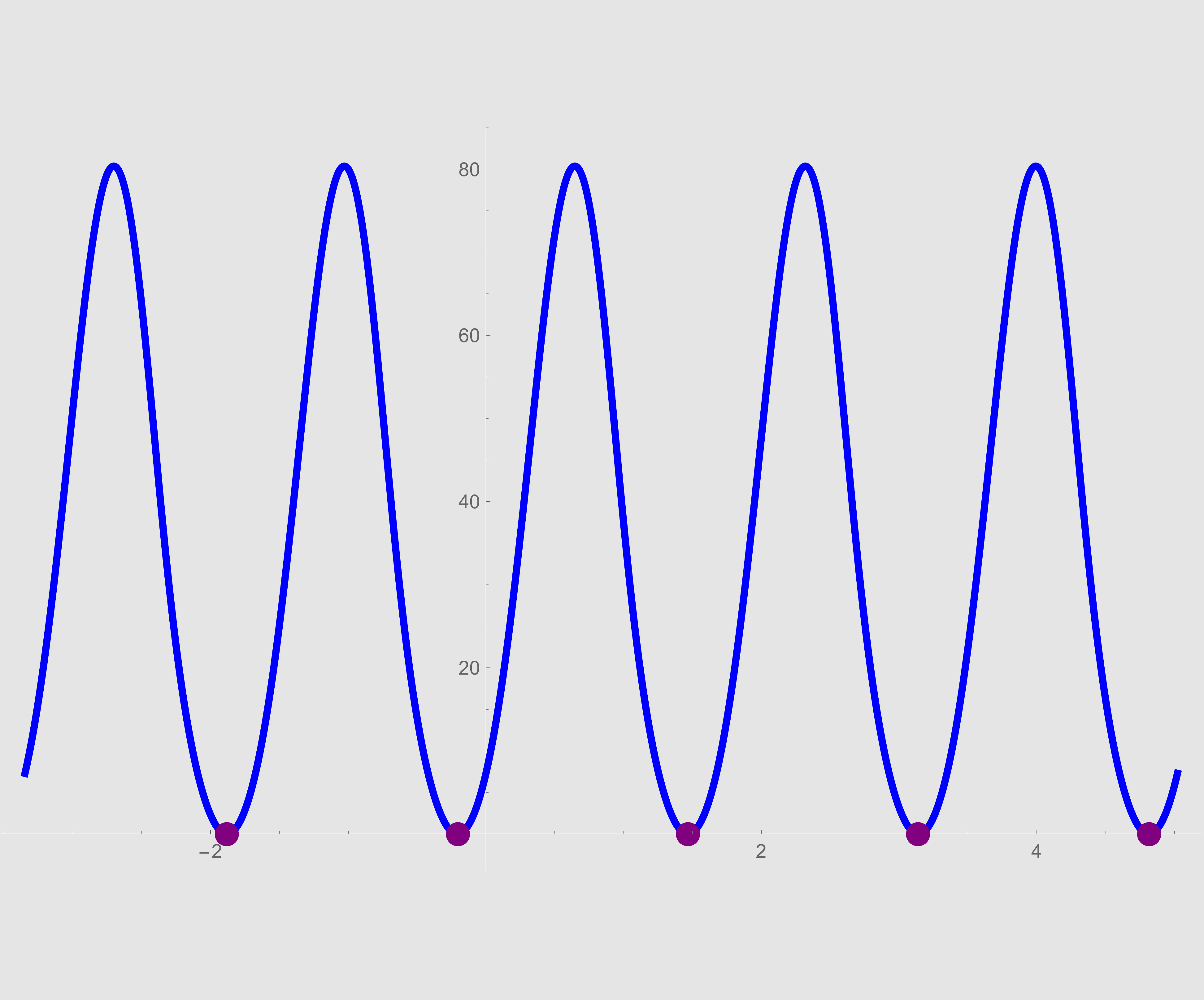}
\includegraphics[height=5cm,width=6cm]{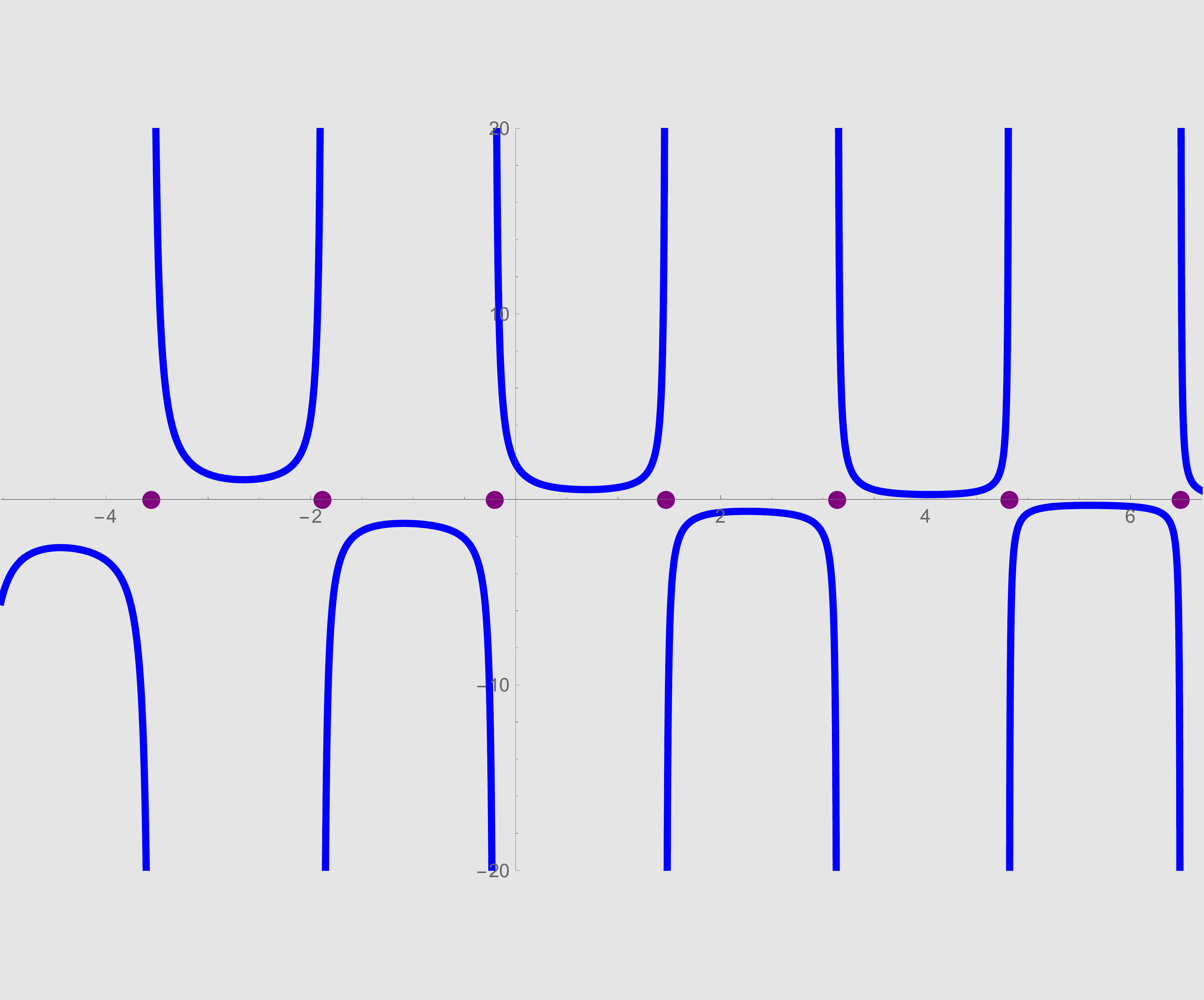}
\caption{The graphs of the functions $\|L_{\lambda}\|^2$ (on the left) and  $e^{-\delta_{\lambda}}$ (on the right), $\lambda\in \R$.}\label{FIG2}
\end{center}
\end{figure}

\begin{remark}{The integrating factors of the first kind are quasi-periodic functions, with quasi-period $2\omega$. If $\lambda\notin \R$, the function $\delta_{\lambda}$ is a regular, complex-valued function. If $\lambda\in \R$, the integrating factor $\delta_{\lambda}$ is real-analytic on the complement of the discrete set $D_{\lambda}$ and its imaginary part is locally constant. The function $e^{\delta_{\lambda}}$ is real-valued, with singularities at the points of $D_{\lambda}$ (see Figure \ref{FIG2}). The singularities of $e^{-\delta_{\lambda}}$ compensate the zeroes of the functions $L^j_{\lambda}$ so that the products $e^{-\delta_{\lambda}}L^j_{\lambda}$ are regular, real-analytic maps (see Figure \ref{FIG3}).
The evaluation of the integrating factors in terms of elliptic integrals and theta functions is analyzed in the appendix. The explicit expression of the integrating factor of the first kind for a non-real eigenvalue is given in (\ref{ifc}) while the integrating factor of a real eigenvalue can be found in (\ref{IIFR}).}\end{remark}

\begin{figure}[ht]
\begin{center}
\includegraphics[height=5cm,width=6cm]{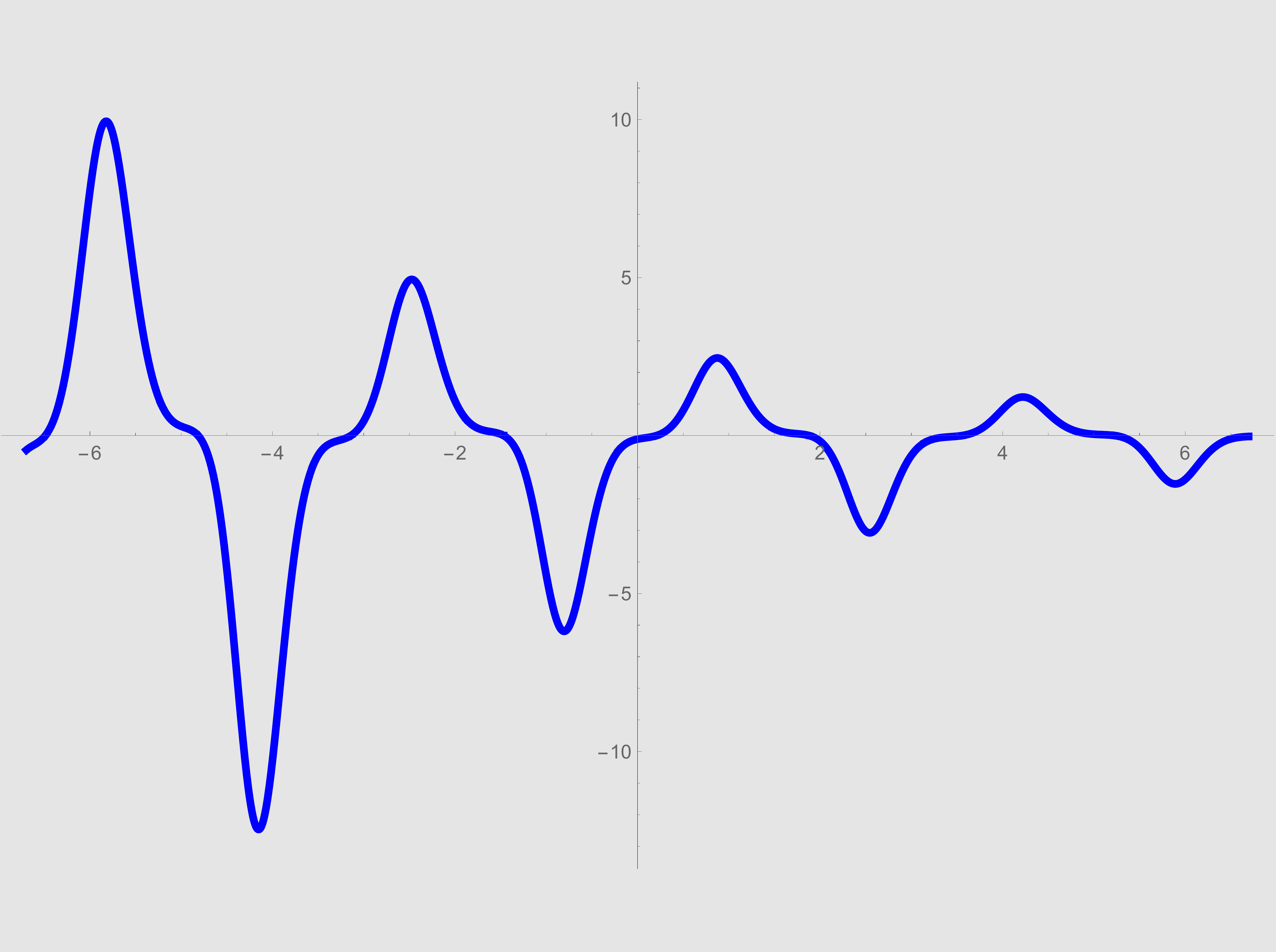}
\includegraphics[height=5cm,width=6cm]{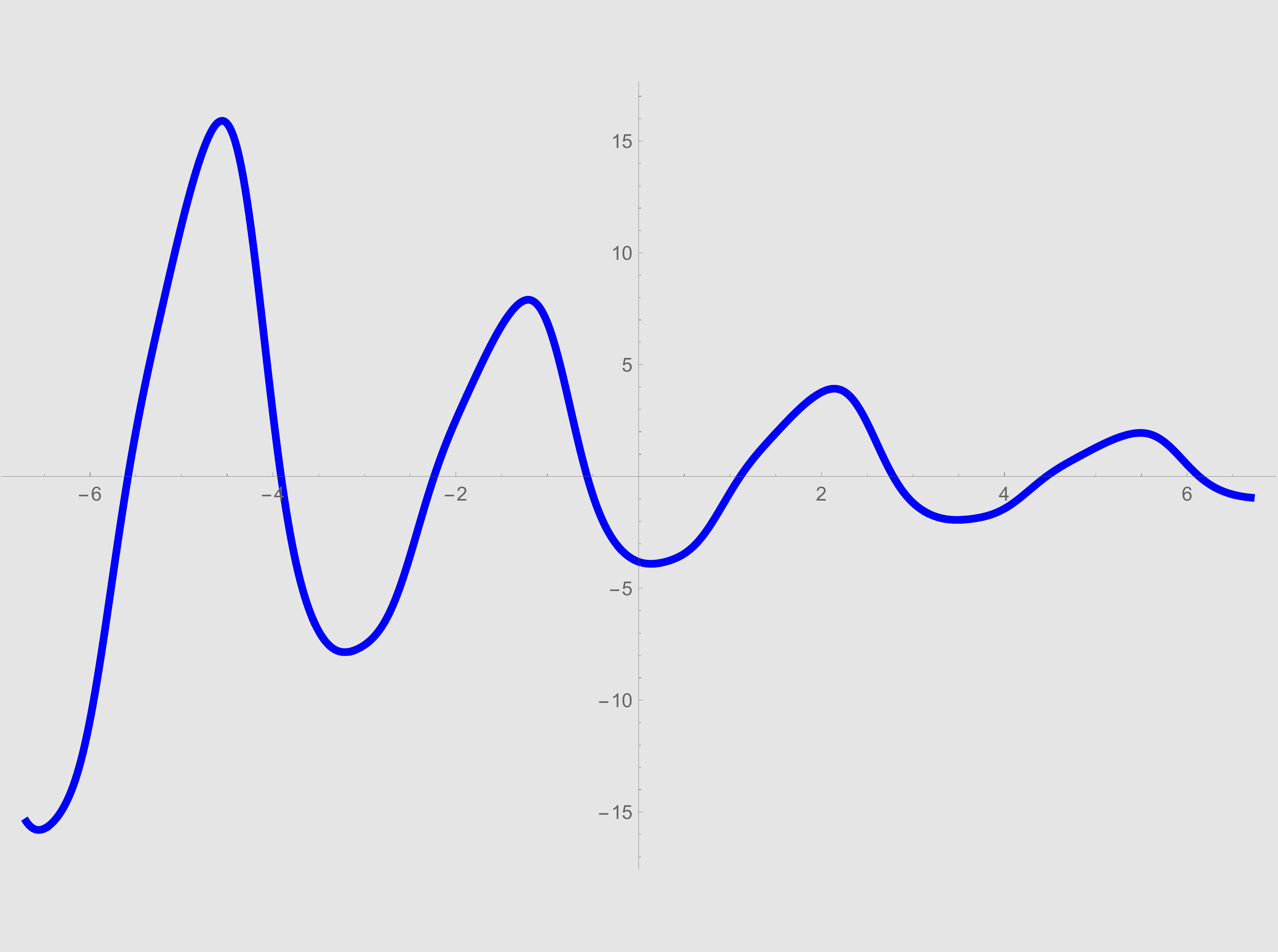}
\caption{The graphs of the functions $e^{-\delta_{\lambda}}L^0_{\lambda}$ (on the left) and $e^{-\delta_{\lambda}}L^3_{\lambda}$ (on the right), $\lambda\in \R$.}\label{FIG3}
\end{center}
\end{figure}

\noindent Let $\lambda$ be a multiple root of $P_\mathfrak{m}$. We set $$\widehat{D}_{\lambda}=\{-\mathrm{sign}(\lambda)p+n\omega : n\in \mathbb{Z}\}$$ and we define
$T_{\lambda}=(T_{\lambda}^0,\dots T_{\lambda}^5):\R-\widehat{D}_{\lambda}\to \R^6$ by
\begin{equation}\label{T}
\begin{cases}
T^0_{\lambda}= \frac{1}{2}(\lambda^2-k_2^2)(6\lambda^2+2c_1+k_2^2),\\ T^1_{\lambda}=\frac{k_2(k_2^2\dot{k}_2+\lambda^2\dot{k}_2-2\lambda k_2)}{\lambda^2-k_2^2},\\
T^2_{\lambda}=-2\lambda(\lambda^2-k_2^2),\\
T^3_{\lambda}=\frac{k_2(\lambda^2+k_2^2-2\lambda k_2 \dot{k}_2)}{\lambda^2-k_2^2},\\
T^4_{\lambda}=0,\\
T_{\lambda}^5=(\lambda^2-k_2^2).
\end{cases}
\end{equation}

\begin{remark}{The map $T_{\lambda}$ is periodic with period $\omega$, is real-analytic on the complement of $\widehat{D}_{\lambda}$ and tends to $\pm \infty$ when $u$ tends to a point of $\widehat{D}_{\lambda}$. It vanishes at the point of $D_{\lambda}$ (see Figure \ref{FIG4})}\end{remark}

\begin{figure}[ht]
\begin{center}
\includegraphics[height=5cm,width=6cm]{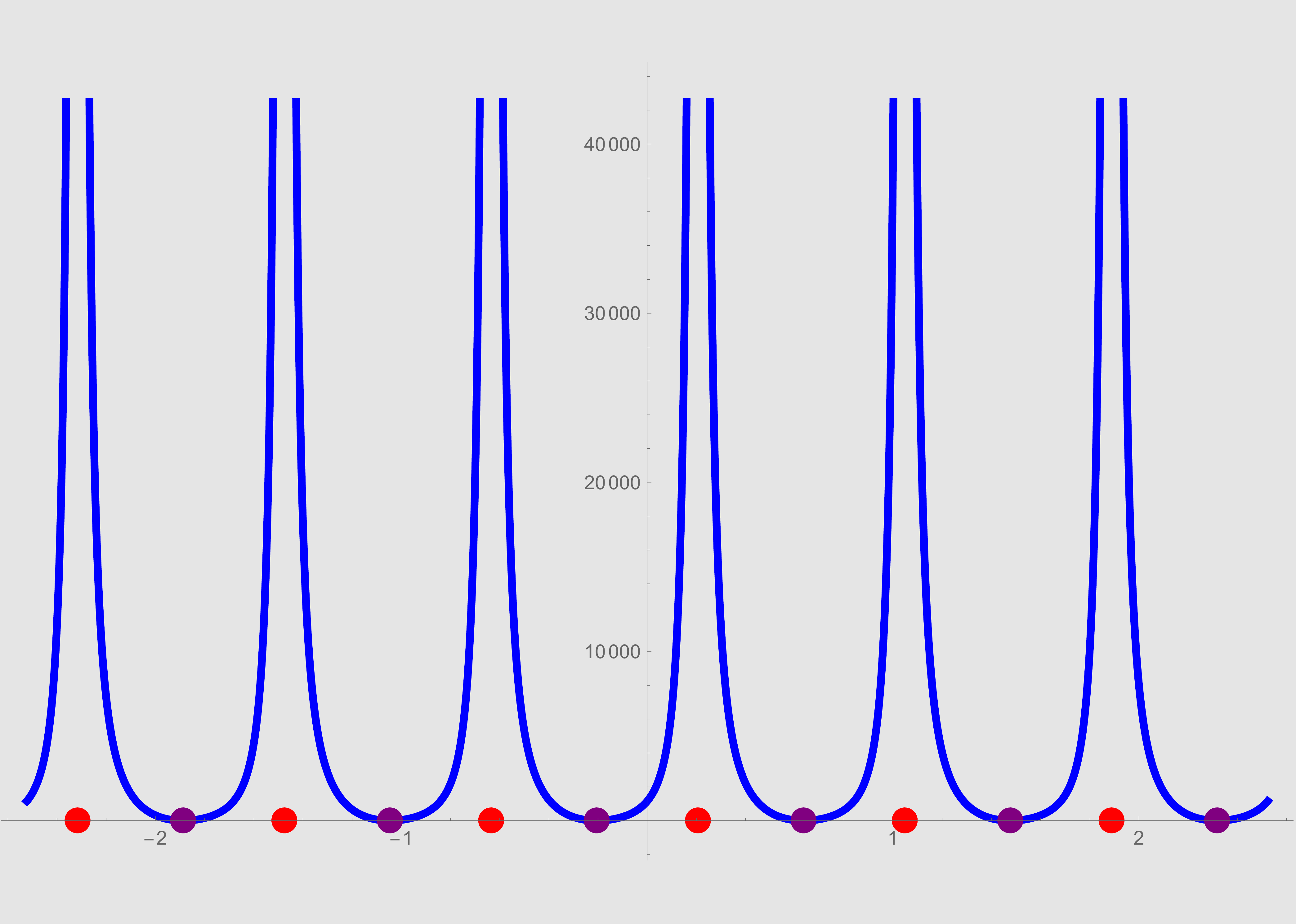}
\includegraphics[height=5cm,width=6cm]{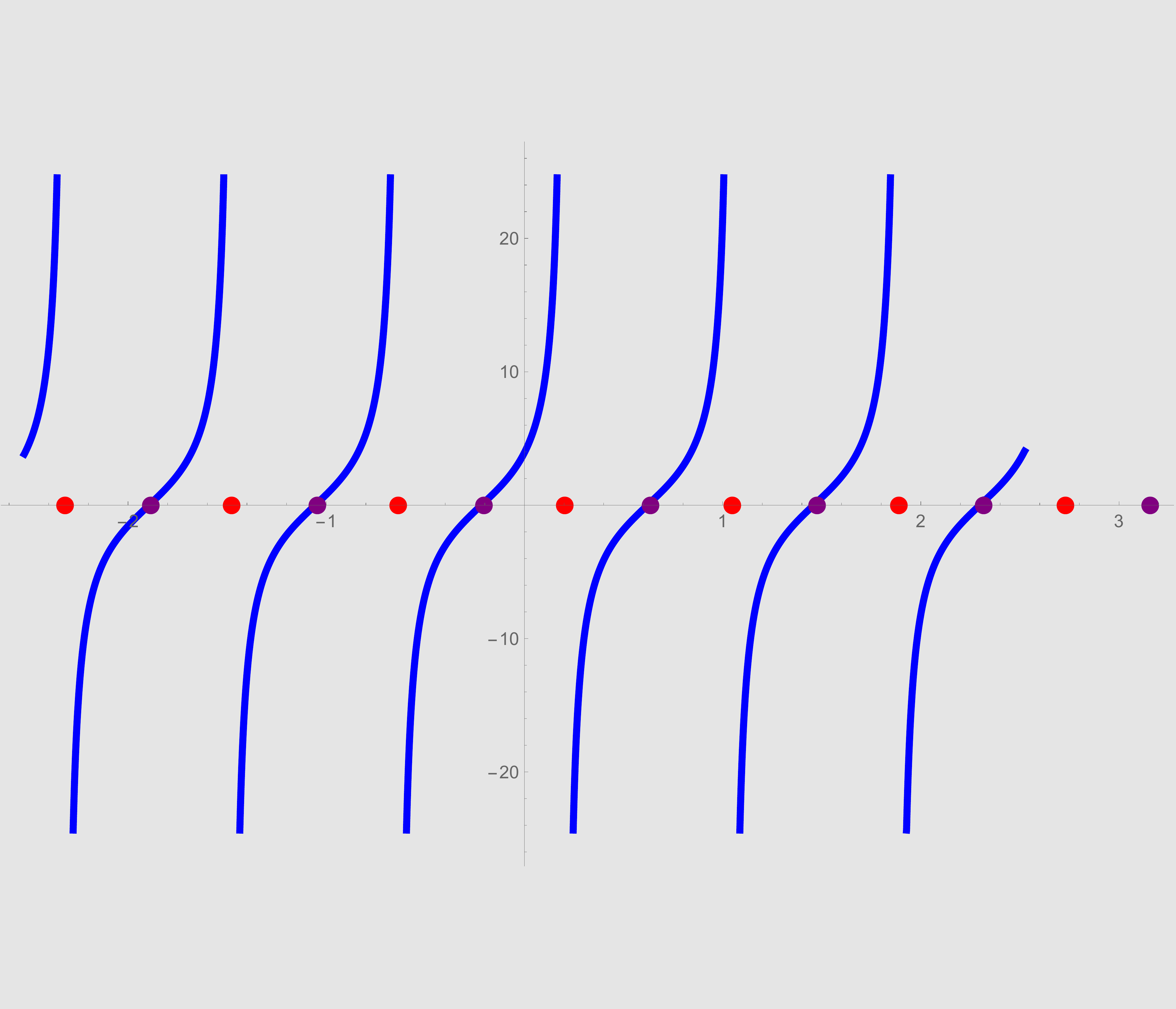}
\caption{The graphs of the functions $\|T_{\lambda}\|^2$ (on the left) and  $\eta_{\lambda}$ (on the right) when $\lambda\in \R$ is a multiple root of $P_\mathfrak{m}$.}\label{FIG4}
\end{center}
\end{figure}

\begin{defn}{A primitive $\eta_{\lambda}:\R-\widetilde{D}_{\lambda}\to \R$ of $s_{\lambda}$ is said an {\it integrating factor of the second kind} for the multiple eigenvalue $\lambda$ if
\begin{itemize}
\item $\eta_{\lambda}|_0=0$;
\item $e^{-\delta_{\lambda}}(T_{\lambda}-\eta_{\lambda}L_{\lambda}):\R-\widetilde{D}_{\lambda}\to \C^6$ extends to a real-analytic map $\R\to \C^6$.
\end{itemize}}\end{defn}

\begin{remark}{The integrating factor of the second kind vanishes at the points of $D_{\lambda}$ and tends to $\pm \infty$ when $u$ tends to a point of $\widehat{D}_{\lambda}$ (see Figure \ref{FIG4}). The functions $\eta_{\lambda}L_{\lambda}$ and $T_{\lambda}-\eta_{\lambda}L_{\lambda}$ behave in a similar way. Multiplying $T_{\lambda}-\eta_{\lambda}L_{\lambda}$ with $e^{-\delta_{\lambda}}$, the zeroes of one factor compensate the singularities of the other so that the product is a regular analytic function (see Figures \ref{FIG5}). The formula expressing the integrating factor of the second kind is given in (\ref{IFIIK}). Despite the apparent opacity, the formulas of the integrating factors can be easily made operative using standard programs of symbolic computation such as {\it Mathematica 11}.}\end{remark}

\begin{figure}[ht]
\begin{center}
\includegraphics[height=5cm,width=6cm]{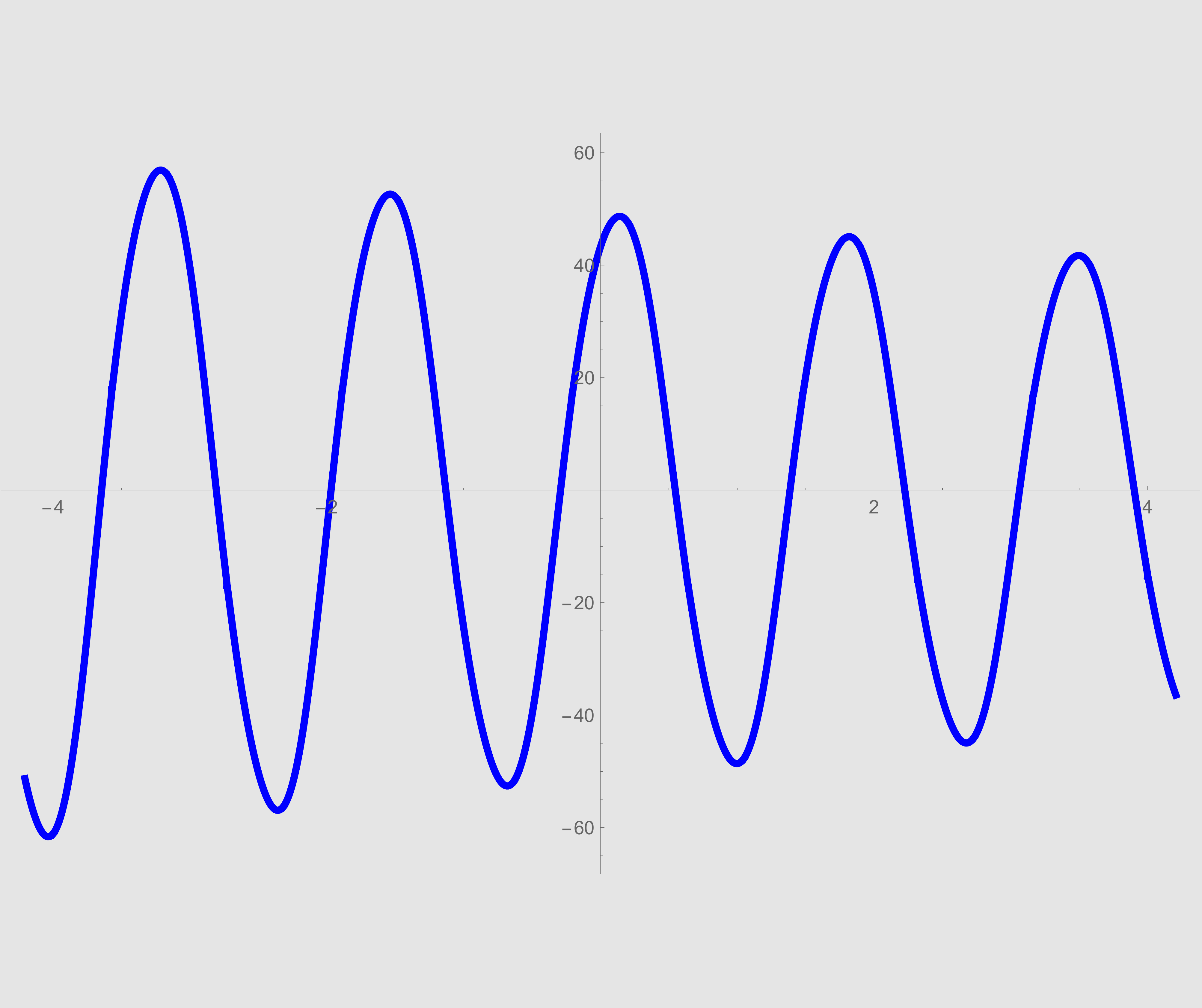}
\includegraphics[height=5cm,width=6cm]{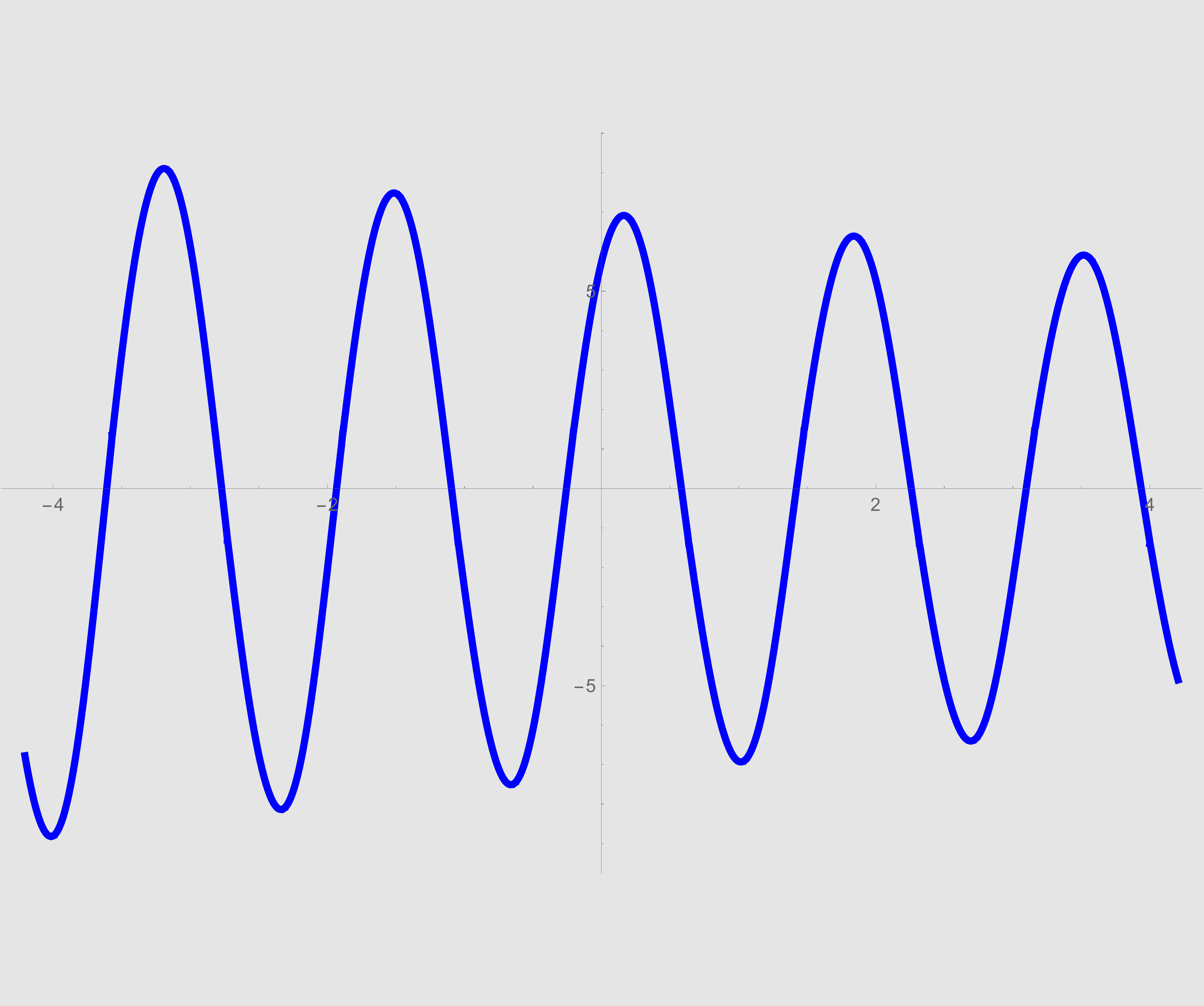}
\caption{The graphs of the functions $e^{-\delta_{\lambda}}(T^0_{\lambda}-\eta_{\lambda}L^0_{\lambda})$ (on the left) and  $e^{-\delta_{\lambda}}(T^3_{\lambda}-\eta_{\lambda}L^3_{\lambda})$ (on the right) when $\lambda\in \R$ is a multiple root of $P_\mathfrak{m}$.}\label{FIG5}
\end{center}
\end{figure}

\begin{prop}\label{PV}{Let $\lambda\in \mathcal{S}_{\mathfrak{m}}$ be an eigenvalue of the momentum, then
\begin{equation}\label{prax1}e^{-\delta_{\lambda}}\sum_{j=0}^{5}L^j_{\lambda}B_j=\mathbf{A}_{\lambda},
\end{equation}
where $\mathbf{A}_{\lambda}\in \C^6$ is an $\mathfrak{m}$-eigenvector of  the eigenvalue $\lambda$. If $\lambda$ is a multiple root of $P_{\mathfrak{m}}$, then
\begin{equation}\label{prax2}e^{-\delta_{\lambda}}\sum_{j=0}^{5}(T^j_{\lambda}-\eta_{\lambda}L^j_{\lambda})B_j
=\mathbf{C}_{\lambda},
\end{equation}
where $\mathbf{C}_{\lambda}\in \C^6$ is a non-zero vector such that
$$\mathbf{C}_{\lambda}\wedge \mathbf{A}_{\lambda}\neq 0,\quad \mathfrak{m}(\mathbf{C}_{\lambda})=\lambda \mathbf{C}_{\lambda}+\mathbf{A}_{\lambda}.$$
We call $\mathbf{A}_{\lambda}$ the {\rm principal vector} of the eigenvalue $\lambda$ and $\mathbf{C}_{\lambda}$ the {\rm secondary principal vector} of the multiple eigenvalue
$\lambda$.}\end{prop}

\begin{proof}\label{PFIF}{Denote by $\mathbf{L}_{\lambda}:\R\to \C^6$ the map
\begin{equation}\label{eigen}\mathbf{L}_{\lambda}=\sum_{i=0}^{5}L_{\lambda}^i(u)B_i|_u.\end{equation}
Let $\mathbb{V}_{\lambda}$ be the $1$-dimensional $\mathfrak{m}$-eigenspace of the eigenvalue $\lambda$. Then, (\ref{momentum}) implies that $\mathbf{L}_{\lambda}|_u\in \mathbb{V}_{\lambda}$, for every $u\in \R$. Therefore, there exist a unique real-analytic map $\widetilde{r}_{\lambda}:\R-D_{\lambda}\to \C$ such that
\begin{equation}\label{if1} \dot{\mathbf{L}}_{\lambda}|_u=\widetilde{r}_{\lambda}|_u\mathbf{L}_{\lambda}|_u\quad \forall u\in \R-D_{\lambda}.\end{equation}
From $\dot{B}=B\cdot \mathcal{K}$ and keeping in mind (\ref{mcc2}) we have
\begin{equation}\label{if11}\dot{\mathbf{L}}_{\lambda}\equiv \frac{k_2\dot{k}_2+\lambda}{k_2^2-\lambda_j^2}\mathbf{L}_{\lambda}\quad \mathrm{mod}(B_0,B_1,B_2,B_3,B_4).\end{equation}
 From (\ref{if1}) and (\ref{if11}) we have $r_{\lambda}=\widetilde{r}_{\lambda}$. Using (\ref{if1}), we deduce that $e^{-\delta_{\lambda}}\mathbf{L}_{\lambda}$ is constant on $\R-D_{\lambda}$. On the other hand, $e^{-\delta_{\lambda}}L_{\lambda}$ extends smoothly across $D_{\lambda}$ and hence also $e^{-\delta_{\lambda}}\mathbf{L}_{\lambda}$ extends to a real-analytic map $\R\to \C^6$. This implies that $e^{-\delta_{\lambda}}\mathbf{L}_{\lambda}=\mathbf{A}_{\lambda}$, for some $\mathbf{A}_{\lambda}\in \mathbb{V}_{\lambda}$. This proves the first part of the statement.

\noindent Let $\lambda$ be a multiple root of $P_{\mathfrak{m}}$. The map $T_{\lambda}$ satisfies
\begin{equation}\label{PT} \mathcal{H}|_u\cdot T_{\lambda}|_u=\lambda T_{\lambda}|_u+L_{\lambda}|_u,\quad
T_{\lambda}|_u\wedge L_{\lambda}|_u\neq 0,\quad \forall u\in \R-D_{\lambda}.
\end{equation}
We put
$$\mathbf{T}_{\lambda}=e^{-\delta_{\lambda}}\sum_{j=0}^{5} T^j_{\lambda}B_j : \R-D_{\lambda}\to \C^6.$$
From (\ref{PT}) we obtain
\begin{equation}\label{PPT}\mathfrak{m}(\mathbf{T}_{\lambda})=\lambda \mathbf{T}_{\lambda} +\mathbf{A}_{\lambda},\quad
\mathbf{T}_{\lambda}\wedge \mathbf{A}_{\lambda}\neq 0.\end{equation}
Differentiating the first equation in (\ref{PPT}) we get $\mathfrak{m}(\mathbf{T}_{\lambda}')=\lambda \mathbf{T}_{\lambda}'$. Thus, there exist a unique real-analytic function $\widetilde{s}_{\lambda}:\R-D_{\lambda}\to \C$ such that $\mathbf{T}_{\lambda}'=\widetilde{s}_{\lambda}\mathbf{A}_{\lambda}$. Using $\dot{B}=B\cdot \mathcal{K}$ we obtain
$$\dot{\mathbf{T}}_{\lambda} \equiv \frac{\lambda^2+k_2^2-2\lambda k_2\dot{k}_2}{(\lambda-k_2^2)^2} \mathbf{A}_{\lambda},\quad \mathrm{mod}(B_0,B_1,B_2,B_3,B_5).$$
Then, $\widetilde{s}_{\lambda}=s_{\lambda}$. This implies that $\mathbf{T}_{\lambda}-\eta_{\lambda}\mathbf{A}_{\lambda}$ is constant on $\R-D_{\lambda}$. Since $e^{-\delta_{\lambda}}(T_{\lambda}-\eta_{\lambda}L_{\lambda})$ extends smoothly across $\widetilde{D}_{\lambda}$, also $\mathbf{T}_{\lambda}-\eta_{\lambda}\mathbf{A}_{\lambda}$ extends to a smooth (real-analytic) map $\R\to \C^6$. Hence there exist $\mathbf{C}_{\lambda}\in \C^6$ such that $\mathbf{T}_{\lambda}-\eta_{\lambda}A_{\lambda}=\mathbf{C}_{\lambda}$. Using (\ref{PPT}) it follows that $\mathfrak{m}(\mathbf{C}_{\lambda})=\lambda \mathbf{C}_{\lambda} + \mathbf{A}_{\lambda}$ and that $\mathbf{C}_{\lambda}\wedge \mathbf{A}_{\lambda}\neq 0$. This yields the result.}\end{proof}

\begin{remark}{Since $\gamma$ is a standard configuration of a world-line then $B|_0=\mathrm{Id}_{6\times 6}$. Hence
$$\mathbf{A}_{\lambda}=L_{\lambda}|_0\quad \mathbf{C}_{\lambda}= T_{\lambda}|_0.$$
Therefore, the principal vectors can be computed explicitly in terms of the phase parameters $e_1,e_2$ and $e_3$.}\end{remark}

\section{Integrability by quadratures}
\subsection{Integrability by quadratures of the regular conformal world-lines} Let $\gamma:\R\to \mathcal{E}^{1,3}$ be the standard configuration of a regular linearly full conformal world-line. Its momentum operator has six simple roots $\lambda_0,\dots ,\lambda_5$, ordered as in (\ref{eigenvl}).
Let $\delta_j$ be the integrating factor of the first kind and
$\mathbf{A}_j$ be the principal vector of $\lambda_j$ respectively. We denote by $\mathbf{A}\in \C(6,6)$ the matrix with column vectors $\mathbf{A}_0,\dots , \mathbf{A}_5$ and we define
\footnote{$\varepsilon_j=^t(\delta^0_j,\dots, \delta^5_j)$ are the column vectors of the canonical basis of $\R^6$.} the real-analytic maps
$$\Delta,\Lambda : \R-D_{\lambda}\to \C(6,6),\quad \mathbf{X}:\R\to \C(6,6)$$
by
$$\begin{cases}
\Delta = (e^{-\delta_0}\varepsilon_0,\dots e^{-\delta_5}\varepsilon_5),\\
\Lambda = (L_{\lambda_0},\dots, L_{\lambda_5}),\\
\mathbf{X}=\Delta\cdot \Lambda.\end{cases}
$$

\begin{thm}\label{ThmC}{Let $\gamma$ be as above, then
$$\gamma =\mathtt{m}\cdot ^t(\mathbf{A}^{-1})\cdot \mathbf{X} \cdot \mathtt{m}\cdot E_0,$$
where $\mathtt{m}$ is the matrix representing the scalar product $\langle-,-\rangle$ with respect to the standard light-cone basis $(E_0,\dots, E_5)$ of $\R^{2,4}$.}\end{thm}
\begin{proof}{Since $\mathbf{A}_j=e^{-\delta_j}B\cdot L_{\lambda_j}$, $j=0,\dots, 5$, we have
$\mathbf{A}=B\cdot \Lambda\cdot \Delta$, that is
\begin{equation}\label{FFF1}B=\mathbf{A}\cdot \Delta^{-1}\cdot \Lambda^{-1}.\end{equation}
From (\ref{FFF1}) and keeping in mind that $^tB\cdot \mathtt{m}\cdot B = \mathtt{m}$ we get
\begin{equation}\label{FFF2}\Lambda=\mathtt{m}\cdot \mathbf{X}^{-1}\cdot ^t\mathbf{A}\cdot \mathtt{m}\cdot \mathbf{A}\cdot \Delta^{-1}.\end{equation}
Inserting (\ref{FFF2}) in (\ref{FFF1}) we have
$$B=\mathtt{m}\cdot ^t(\mathbf{A}^{-1})\cdot \mathbf{X}\cdot \mathtt{m}.$$
This yields the result.
}\end{proof}

\subsection{Integrability by quadratures of the exceptional conformal world-lines} Let $\gamma:\R\to \mathcal{E}^{1,3}$ be the standard configuration of an exceptional linearly full conformal world-line. Its momentum operator has four distinct roots $\lambda_0,\dots ,\lambda_3$, ordered as in (\ref{eigenvl}). Then, $\lambda_0$ and $\lambda_1$ are simple purely imaginary roots and $\lambda_2,\lambda_3$ are real double roots of $P_{\mathfrak{m}}$. For each eigenvalue $\lambda_j$, $j=0,\dots, 3$, let $\delta_j$ be its integrating factor of the first kind and $\mathbf{A}_j$ be the corresponding principal vector. For each double root $\lambda_j$, $j=2,3$, let $\eta_j$ be the integrating factor of the second kind and $\mathbf{C}_j$ be the secondary principal vector. We denote by $\widetilde{\mathbf{A}}$ the matrix $\widetilde{\mathbf{A}}=
(\mathbf{A}_0,\mathbf{A}_1,\mathbf{A}_2,\mathbf{A}_3,\mathbf{C}_2,\mathbf{C}_3)$. Let
$$\widetilde{\Delta},\widetilde{\Lambda}:\R-D_{\lambda}\to \C(6,6),\quad
\widetilde{\mathbf{X}}:\R\to \C(6,6)
$$
be the real-analytic maps defined by
$$\begin{cases}
\widetilde{\Delta}=(e^{-\delta_0}\varepsilon_0,e^{-\delta_1}\varepsilon_1,
    e^{-\delta_2}\varepsilon_2,e^{-\delta_3}\varepsilon_3,e^{-\delta_2}\varepsilon_4,
    e^{-\delta_3}\varepsilon_5),\\
\widetilde{\Lambda}=(L_{\lambda_0},L_{\lambda_1},L_{\lambda_2},L_{\lambda_3},
T_{\lambda_2}-\eta_2L_{\lambda_2},T_{\lambda_3}-\eta_3 L{\lambda_3}),\\
\widetilde{\mathbf{X}}=\widetilde{\Delta}\cdot \widetilde{\Lambda}
\end{cases}
$$

\begin{thm}\label{ThmD}{Let $\gamma$ be as above, then
$$\gamma =\mathtt{m}\cdot ^t(\widetilde{\mathbf{A}}^{-1})\cdot \widetilde{\mathbf{X}}\cdot \mathtt{m}\cdot E_0.$$}\end{thm}
\begin{proof}{From Proposition \ref{PFIF} it follows that the canonical frame, the principal vectors and the integrating factors satisfy
$$e^{-\delta_0}B\cdot L_{\lambda_0}=\mathbf{A}_0,\quad e^{-\delta_1}B\cdot L_{\lambda_1}=\mathbf{A}_1,\quad e^{-\delta_2}B\cdot L_{\lambda_2}=\mathbf{A}_2,\quad e^{-\delta_3}B\cdot L_3=\mathbf{A}_3$$
and
$$e^{-\delta_2}B\cdot (T_{\lambda_2}-\eta_2L_{\lambda_w})=\mathbf{C}_2\quad e^{-\delta_3}B\cdot (T_{\lambda_3}-\eta_3L_{\lambda_3})=\mathbf{C}_3.$$
We then have
\begin{equation}\label{FFFF1}B=\widetilde{\mathbf{A}}\cdot \widetilde{\Delta}^{-1}\cdot \widetilde{\Lambda}^{-1}.\end{equation}
Combining (\ref{FFFF1}) with $^tB\cdot \mathtt{m}\cdot B= \mathtt{m}$ we obtain
\begin{equation}\label{FFFF2}\widetilde{\Lambda}^{-1}=\widetilde{\Delta}\cdot \widetilde{\mathbf{A}}^{-1}\cdot \mathtt{m}\cdot ^t(\widetilde{\mathbf{A}}^{-1})\cdot \widetilde{\mathbf{X}}\cdot \mathtt{m}.\end{equation}
Substituting (\ref{FFFF2}) into (\ref{FFFF1}) we find $B=\mathtt{m}\cdot ^t(\widetilde{\mathbf{A}}^{-1})\cdot \widetilde{\mathbf{X}}\cdot \mathtt{m}$. We have thus proved the result.}\end{proof}

\subsection{Final comments}
The theoretical explanation of the integrability by quadratures lies in the Arnold-Liouville integrability of the Euler-Lagrange differential system. With this we mean the following : let $\mathfrak{m}\in \mathfrak{a}(2,4)$ be the momentum of a linearly full world-line $\gamma$ with non-constant curvatures, then $\mathfrak{m}$ is either regular or exceptional. The stabilizer $A^{\uparrow}_+(2,4)_{\mathfrak{m}}$ of $\mathfrak{m}$ is a $3$-dimensional closed subgroup, diffeomorphic to $S^1\times \R^2$. The inverse image of $\mathfrak{m}^{\natural}$ by the momentum map is a four-dimensional sub-manifold $Z_{\mathfrak{m}}\subset Z$ and the characteristic vector field $X_{\zeta}$ is tangent to $Z_{\mathfrak{m}}$. The world-lines with momentum $\mathfrak{m}$ are originated by the integral curves of $X_{\zeta}|_{Z_{\mathfrak{m}}}$.  The stabilizer $A^{\uparrow}_+(2,4)_{\mathfrak{m}}$ acts freely on $Z_{\mathfrak{m}}$ and the quotient space $Z_{\mathfrak{m}}/A^{\uparrow}_+(2,4)_{\mathfrak{m}}$ is a circle. This implies that $Z_{\mathfrak{m}}\subset Z$ is diffeomorphic to the Cartesian product of $\R^2$ with a $2$-dimensional torus $T^2$. In principle, the integration by quadratures  can be achieved by a diffeomorphism $\Psi_{\mathfrak{m}} : Z_{\mathfrak{m}}\to \R^2\times T^2$ such that $\Psi_*(X_{\zeta})$ is a linear vector field. Since the stabilizer of the momentum operator is non compact then the trajectory of $\gamma$ can't be closed. Instead, if $\gamma$ is trapped in a $3$-dimensional Einstein universe, the stabilizer of the momentum operator can be a maximal compact abelian subgroup of $A^{\uparrow}_+(2,3)$. Thus, in this case, there are countably many closed world-lines with non-constant curvatures, as it has been shown in \cite{DMN}.

\noindent  The periodicity of the conformal curvatures implies that  its trajectory of $\gamma$ is left unchanged by the action of the infinite cyclic subgroup generated by $B(\omega)\cdot B(0)^{-1}\in A^{\uparrow}_+(2,4)$.

\noindent Lastly, we note that if we know the phase parameters $e_1,e_2,e_3$ all the steps of the integration procedure can be implemented in {\it Mathematica} and are fully operative from a computational viewpoint.


\section{Appendix : the integrating factors}

\subsection{Integrating factors of the first kind} Given $\lambda\in \mathcal{S}_{\mathfrak{m}}$ we put
\begin{equation}\label{abc}a=\ell_1-\lambda^2\ell_3,\quad b=\ell_2-\lambda\ell_4,\quad c=\frac{\lambda\ell_4}{\ell_2-\lambda^2\ell_4}\end{equation}
and
\begin{equation}\label{d}d=\frac{\lambda (\ell_2\ell_3 -\ell_1 \ell_4)}
{(\ell_2-\lambda^2\ell_4)(\ell_1-\lambda^2\ell_3)},
\end{equation}
where $\ell_1,\ell_2,\ell_3$ and $\ell_4$ are the constants defined as in (\ref{ell}). From (\ref{k21}) and (\ref{if2}) we obtain
\begin{equation}\label{r}r_{\lambda}(u)=\frac{1}{2}\frac{d}{du}\left(\ln \frac{a-b\mathrm{sn}^2(\sqrt{\ell_3}u,m)}{\ell_3-\ell_4 \mathrm{sn}^2(\sqrt{\ell_3}u,m)}\right)+c+\frac{d}{1-\alpha^2\mathrm{sn}^2(\sqrt{\ell_3}u,m)},
\end{equation}
where the parameter $m$ is as in (\ref{ell}) and $\alpha$ is as in (\ref{palpha}).

\subsubsection{The integrating factor of the first kind of a non-real eigenvalue}
Let $\lambda$ be a non-real eigenvalue, then
$$1-\alpha^2\mathrm{sn}^2(\sqrt{\ell_3}u,m)\neq 0,\quad \forall u\in \R$$
and
$$\int \frac{du}{1-\alpha^2\mathrm{sn}^2(u,m)}=\Pi(\alpha^2,\mathrm{am}_m(u),m),$$
where $\Pi(n,\phi,m)$ is the incomplete integral of the third kind and $\mathrm{am}_m(-)$ is the Jacobi's amplitude with parameter $m$. Note that in this case, the restriction of the incomplete integral of the third kind on the real axis is a regular real-analytic function. Since
$$\frac{a-b\mathrm{sn}^2(\sqrt{\ell_3}u,m)}{\ell_3-\ell_4 \mathrm{sn}^2(\sqrt{\ell_3}u,m)}\notin \R_-,\quad \forall u\in \R$$
we can evaluate the logarithm\footnote{We use the standard determination of the natural logarithm, with a branch cut discontinuity in the complex plane running from $-\infty$ to $0$.} of the function on the left hand side in the above formula. Thus, the integrating factor of the eigenvalue is given by
\begin{equation}\label{ifc}
\delta_{\lambda}(u)=\frac{1}{2}\ln\left(\frac{a-b\mathrm{sn}^2(\sqrt{\ell_3}u,m)}{\ell_3-\ell_4 \mathrm{sn}^2(\sqrt{\ell_3}u,m)}\right)+c u+\frac{d}{\sqrt{\ell_3}} \Pi(\alpha^2,\mathrm{am}_m(\sqrt{\ell_3} u),m).
\end{equation}

 \subsubsection{The integrating factor of the first kind of a real eigenvalue}\label{IFFR} If $\lambda$ is a real eigenvalue, the function $r_{\lambda}$ is singular and the evaluation of the integrating factor requires some caution.  Let $w$ and $v$ be the real constants
$$w=\frac{\alpha}{\sqrt{(\alpha^2-m)(\alpha^2-1)}},$$
and
$$v=\frac{E(m)}{K(m)}-E(p,m)-\mathrm{cs}(p,m)\mathrm{dn}(p,m)-
\frac{\sqrt{(\alpha^2-m)(\alpha^2-1)}}{\alpha},$$
where $p=p_{\lambda}$ is as in (\ref{palpha}) and $E(m), E(-,m)$ are the complete and incomplete elliptic integrals of the second kind respectively. Let $f_{\lambda}$ be the periodic extension with period $2\omega$ of the locally constant function
$$
f_{\lambda}(u)=\begin{cases}-\frac{\pi}{2},\quad u\in [p-\omega,p),\\
-\frac{3\pi}{2},\quad u\in [p,\omega-p),\quad \quad \lambda <0 \\
\quad \frac{\pi}{2},\quad u\in [\omega-p,\omega+p),\end{cases}
$$
and
$$
f_{\lambda}(u)=\begin{cases}-\frac{\pi}{2},\quad u\in [p-\omega,p),\\
\quad \frac{\pi}{2},\quad u\in [p,\omega+p),\end{cases}
$$
if $\lambda>0$. Denote by $\vartheta_1(-,q_m)$ the first Jacobi's theta function with nome
$$q_m=\mathrm{exp}(-\pi K(1-m)/K(m)).$$
Proceeding as in \cite{La}, p.$71$, we see that
\begin{equation}\label{PTR1}
g_{\lambda,1}(u)=\frac{w}{2\sqrt{\ell_3}}\ln
\left(\frac{\vartheta_1(\frac{\pi}{2K(m)}(p-
\frac{u}{\sqrt{\ell_3}}),q_m)}{\vartheta_1(\frac{\pi}{2K(m)}(p+\frac{u}{\sqrt{\ell_3}}),q_m)} \right)+wvu
\end{equation}
is a real-valued primitive of $(1-\alpha^2\mathrm{sn}^2(\sqrt{\ell_3}u,m))^{-1}$. We take
\begin{equation}\label{PTR2}
\begin{split}
g_{\lambda,2}(u) & =\frac{1}{2}\ln(a)+cu-\frac{1}{2}\ln(\ell_3-
\ell_4\mathrm{sn}(\sqrt{\ell_3}u,m)+\\
& + \frac{1}{2}\ln(1-\alpha^2\mathrm{sn}^2(\sqrt{\ell_3}u,m))
\end{split}
\end{equation}
as a real-valued primitive of
$$\frac{1}{2}\frac{d}{du}\left(\ln \frac{a-b\mathrm{sn}^2(\sqrt{\ell_3}u,m)}{\ell_3-\ell_4 \mathrm{sn}^2(\sqrt{\ell_3}u,m)}\right)+c.$$
Then,
\begin{equation}\label{IIFR}\delta_{\lambda}= dg_{\lambda,1}+g_{\lambda,2}+if_{\lambda},\end{equation}
is the integrating factor for the real eigenvalue $\lambda$.

\subsection{Integrating factors of the second kind} Let $\lambda$ be a multiple root of $P_{\mathfrak{m}}$.
Note that $\lambda$ is necessarily real. From
$$s_{\lambda}=\frac{\lambda^2+k_2^2-2\lambda k_2\dot{k}_2}{(\lambda^2-k_2^2)^2}
=-2\lambda \frac{k_2\dot{k}_2}{(\lambda^2-k_2^2)^2}+\frac{\lambda^2+k_2^2}{k_2^2-\lambda^2}$$
we can write
$$\int s_{\lambda}du = \eta_{1,\lambda}+\eta_{2,\lambda}$$
where
\begin{equation}\label{eta1}
\begin{split}\eta_{1,\lambda}&=-2 \lambda\int \frac{k_2\dot{k}_2}{(\lambda^2-k_2^2)^2}du=
\frac{\lambda}{k_2^2-\lambda^2}=\\
&= \lambda \frac{\ell_3-\ell_4\mathrm{sn}^2(\sqrt{\ell_3}u,m)}{(\ell_1-\lambda^2\ell_3)-
(\ell_2-\lambda^2\ell_4)\mathrm{sn}^2(\sqrt{\ell_3}u,m)}=\\
&=\frac{\lambda}{\ell_1-\lambda^2\ell_3}\frac{\ell_3-\ell_4\mathrm{sn}^2(\sqrt{\ell_3}u,m)}
{1-\alpha^2\mathrm{sn}^2(\sqrt{\ell_3}u,m)}
\end{split}\end{equation}
and
\begin{equation}\label{FTR}
\begin{split}\eta_{2,\lambda}& =\int\frac{\lambda^2+k_2^2}{(\lambda^2-k_2^2)^2}du = \textsf{A} \int \frac{\mathrm{cn}^2(\sqrt{\ell_3}u,m)}{1-\alpha^2 \mathrm{sn}^2(\sqrt{\ell_3}u,m)}du+\\
&+\textsf{B} \int \frac{\mathrm{cn}^2(\sqrt{\ell_3}u,m)\mathrm{sn}^2(\sqrt{\ell_3}u,m)}{1-\alpha^2 \mathrm{sn}^2(\sqrt{\ell_3}u,m)}du + \textsf{C} \int \frac{\mathrm{sn}^2(\sqrt{\ell_3}u,m)}{1-\alpha^2 \mathrm{sn}^2(\sqrt{\ell_3}u,m)}du,
\end{split}
\end{equation}
where
$$
\textsf{A}=\frac{\ell_3(\ell_1+\lambda^2\ell_3)}{(\ell_1-\lambda^2\ell_3)^2},\quad
\textsf{B}=-\frac{\ell_4(\ell_2+\lambda_2\ell_4)}{(\ell_1-\lambda^2\ell_3)^2},$$
and
$$
\textsf{C}=\frac{(\ell_3-\ell_4)((\ell_1-\ell_2)+\lambda^2(\ell_3-\ell_4))}{(\ell_1-\lambda^2\ell_3)^2}.
$$
The integrals in the right hand side of (\ref{FTR}) can be evaluated as in \cite{BF} p. $218$, and, as a result, we obtain
\begin{equation}\label{eta2}\begin{split}
\eta_{2,\lambda}&=\frac{\textsf{M}}{\sqrt{\ell_3}}E(\sqrt{\ell_3}u,m)+\textsf{N}u+
\textsf{P}(g_{\lambda,1}-i\frac{\pi}{2}\widehat{f}_{\lambda})+\\
&+\frac{\textsf{Q}}{\sqrt{\ell_3}}
\frac{\mathrm{sn}(\sqrt{\ell_3}u,m)\mathrm{cn}(\sqrt{\ell_3}u,m)\mathrm{dn}(\sqrt{\ell_3}u,m)}
{1-\alpha^2\mathrm{sn}^2(\sqrt{\ell_3}u,m)}
\end{split}
\end{equation}
where $g_{\lambda,1}$ is as in (\ref{PTR1}) and $\widehat{f}_{\lambda}$ is the periodic extension with period $\omega$ of the locally constant function
\[
f(u)=\begin{cases}
 0,\quad u\in [0,p),\\
1,\quad u\in [p,\omega)\end{cases}
\]
and $\textsf{M},\textsf{N},\textsf{P},\textsf{Q}$ are the constants
$$\begin{cases}
\textsf{M}=\frac{\alpha^2}{2\alpha^2(m-\alpha^2)}\left(\textsf{A}+\frac{1}{\alpha^2}\textsf{B}+
\frac{1}{\alpha^2-1}\textsf{C}\right),\\
\textsf{N}=\frac{1}{2\alpha^2}\left(\textsf{A}+\textsf{C}-\frac{1}{\alpha^2}\textsf{B}\right),\\
\textsf{P}=\frac{1}{2\alpha^2(m-\alpha^2)}\left((2m\alpha^2-\alpha^4-m)\textsf{A}+(\alpha^4-m)\textsf{C}+
(\alpha^4-2\alpha^2+m)\textsf{B}\right),\\
\textsf{Q}=-\frac{\alpha^2}{2(m-\alpha^2)}\left(\textsf{A}+\frac{1}{\alpha^2-1}\textsf{C}+
\frac{1}{\alpha^2}\textsf{B}\right).
\end{cases}$$
Summing up : the integrating factor of the second kind of a multiple root is given by
\begin{equation}\label{IFIIK}\eta_{\lambda}=\eta_{1,\lambda}+\eta_{2,\lambda}
\end{equation}
where $\eta_{1,\lambda}$ and $\eta_{2,\lambda}$ are the functions defined as in (\ref{eta1}) and (\ref{eta2}).
\bibliographystyle{amsalpha}

\end{document}